\documentclass[12pt]{amsart}

\usepackage[top=3cm,bottom=3cm,left=3cm,right=3cm]{geometry}

\usepackage{amsmath}
\usepackage{amssymb}
\usepackage{amsthm}
\usepackage{latexsym}
\usepackage{mathtools}
\usepackage{thmtools}
\usepackage{amsfonts}
\usepackage{mathrsfs}
\usepackage{textcomp}
\usepackage{graphicx}
\usepackage{setspace}
\usepackage{nicefrac}
\usepackage{indentfirst}
\usepackage{enumerate}
\usepackage{wasysym}
\usepackage[normalem]{ulem}
\usepackage{upgreek}
\usepackage{paralist}
\usepackage{xcolor}
\usepackage{etoolbox}
\usepackage{mathdots}
\usepackage{wrapfig}
\usepackage{floatflt}
\usepackage{tensor} 
\usepackage{parskip}
\usepackage{lscape}
\usepackage{stmaryrd}
\usepackage[enableskew]{youngtab}
\usepackage{ytableau}
\usepackage{pifont}
\usepackage{braket}

\usepackage{nicematrix} 

\usepackage{scalerel}

\usepackage[all,cmtip]{xy}
\usepackage[labelfont=bf, font={small}]{caption}[2005/07/16]

\usepackage{xcolor}
\definecolor{red}{rgb}{1,0,0}

\newcommand{\xto}[1]{\xrightarrow{\phantom{a}{#1}{\phantom{a}}}}

\newcommand{\vvirg}{ , \dots , }
\newcommand{\ootimes}{ \otimes \cdots \otimes }
\newcommand{\ooplus}{ \oplus \cdots \oplus }
\newcommand{\ttimes}{ \times \cdots \times }

\newcommand{\contract}{\mathchoice
  {\rotatebox[origin=c]{180}{ \reflectbox{$\neg$} }}
  {\rotatebox[origin=c]{180}{ \reflectbox{$\neg$} }}
  {\scalebox{.7}{\rotatebox[origin=c]{180}{ \reflectbox{$\neg$} }}}
  {\scalebox{.5}{\rotatebox[origin=c]{180}{ \reflectbox{$\neg$} }}}
}

\newcommand{\bfn}{\mathbf{n}}

\newcommand{\bft}{\mathbf{t}}

\newcommand{\calB}{\mathcal{B}}

\newcommand{\calH}{\mathcal{H}}

\newcommand{\calL}{\mathcal{L}}

\newcommand{\calO}{\mathcal{O}}

\newcommand{\calS}{\mathcal{S}}

\newcommand{\bbC}{\mathbb{C}}

\newcommand{\bbN}{\mathbb{N}}

\newcommand{\bbP}{\mathbb{P}}

\newcommand{\bbT}{\mathbb{T}}

\newcommand{\bbZ}{\mathbb{Z}}

\newcommand{\frakS}{\mathfrak{S}}

\newcommand{\frakg}{\mathfrak{g}}

\renewcommand{\phi}{\varphi}

\renewcommand{\tilde}[1]{\widetilde{#1}}
\renewcommand{\hat}[1]{\widehat{#1}}
\renewcommand{\bar}[1]{\overline{#1}}

\newcommand{\id}{\mathrm{id}}
\newcommand{\Id}{\mathrm{Id}}

\newcommand{\rank}{\mathrm{rank}}
  
\newcommand{\image}{\mathrm{image}}  
\newcommand{\Gr}{\mathrm{Gr}}

\DeclareMathOperator{\End}{End}

\DeclareMathOperator{\Sym}{Sym}
\DeclareMathOperator{\Aut}{Aut}
\DeclareMathOperator{\Inn}{Inn}
\DeclareMathOperator{\Out}{Out}

\DeclareMathOperator{\Pic}{Pic}

\newcommand{\GL}{\mathrm{GL}}
\newcommand{\SL}{\mathrm{SL}}
\newcommand{\PGL}{\mathrm{PGL}}

\newcommand{\frakgl}{\mathfrak{gl}}

\newcommand{\Sub}{\mathrm{Sub}}

\renewcommand{\det}{\mathrm{det}}

\newcommand{\conv}{\mathrm{conv}}

\newcommand{\Part}{\mathrm{Part}}
\newcommand{\Sing}{\mathrm{Sing}}
\newcommand{\rmCH}{\mathrm{CH}}

\def\ba#1\ea{\begin{align}#1\end{align}}
\newcommand{\setft}[1]{\mathrm{#1}}

\usepackage[pagecolor=white]{pagecolor}
\usepackage{xcolor}
\makeatletter
\newcommand{\globalcolor}[1]{%
  \color{#1}\global\let\default@color\current@color
}
\makeatother

\AtBeginDocument{\globalcolor{black}}

\usepackage[colorlinks]{hyperref}
\hypersetup{
  linkcolor=[rgb]{0.3,0.3,0.6},
  citecolor=[rgb]{0.2, 0.6, 0.2},
  urlcolor=[rgb]{0.6, 0.2, 0.2}
}

\newtheorem{theorem}{Theorem}[section]
\newtheorem{lemma}[theorem]{Lemma}
\newtheorem{proposition}[theorem]{Proposition}
\newtheorem{corollary}[theorem]{Corollary}

\theoremstyle{definition}

\newtheorem{remark}[theorem]{Remark}
\newtheorem{example}[theorem]{Example}

\title[Linear preservers]{Linear preservers of secant varieties \\ and other varieties of tensors}
\author[F. Gesmundo, Y. Han, B. Lovitz]{Fulvio Gesmundo, Young In Han, Benjamin Lovitz}

\address[F. Gesmundo]{Institut de Mathématiques de Toulouse; UMR5219 -- Université de Toulouse; CNRS -- UPS, F-31062 Toulouse Cedex 9, France}
\email{fgesmund@math.univ-toulouse.fr}

\address[Y. Han]{University of Waterloo, Waterloo, Ontario, Canada}
\email{youngin.daniel.han@gmail.com}

\address[B. Lovitz]{Northeastern University, Boston, Massachusetts, USA}
\email{benjamin.lovitz@gmail.com}



\newcommand\blfootnote[1]{%
  \begingroup
  \renewcommand\thefootnote{}\footnote{#1}%
  \addtocounter{footnote}{-1}%
  \endgroup
}

\begin{document}

\begin{abstract}
We study the problem of characterizing linear preserver subgroups of algebraic varieties, with a particular emphasis on secant varieties and other varieties of tensors. We introduce a number of techniques built on different geometric properties of the varieties of interest. Our main result is a simple characterization of the linear preservers of secant varieties of Segre varieties in many cases, including $\sigma_r((\bbP^{n-1})^{\times k})$ for all $r \leq n^{\lfloor k/2 \rfloor}$.
We also characterize the linear preservers of several other sets of tensors, including subspace varieties, the variety of slice rank one tensors, symmetric tensors of bounded Waring rank, the variety of biseparable tensors, and hyperdeterminantal surfaces. Computational techniques and applications in quantum information theory are discussed. We provide geometric proofs for several previously known results on linear preservers.
\end{abstract}

\maketitle

\section{Introduction} \blfootnote{ \hrule ~ \\ 
Fulvio Gesmundo (corresponding author). Institut de Mathématiques de Toulouse; UMR5219 -- Université de Toulouse; CNRS -- UPS, F-31062 Toulouse Cedex 9, France  \\
Young In Han. University of Waterloo, Waterloo, Ontario, Canada \\
Benjamin Lovitz. Northeastern University, Boston, Massachusetts, USA 
}
Given a group $G$ acting on a set $V$ and a subset $X \subseteq V$, let
\[
G_X: = \{ g \in G : g(x) \in X \text{ for every $x \in X$}\}
\]
be the \textit{preserver of $X$}; this is the largest subgroup of $G$ mapping $X$ to itself. Determining $G_X$ is useful in many contexts for characterizing the symmetries of a problem~\cite{GUTERMAN200061,li2001linear,tan2003determinant,molnar2007selected,johnston2011characterizing,contreras2019resource,oneto2023hadamard}.

In this work, we focus on the setting where $V$ is a complex vector space and $G:=\GL(V)$ is the general linear group, wherein characterizing $G_X$ is commonly known as a \textit{linear preserver problem}. We present various techniques for determining $G_X$ when $X$ is an algebraic subvariety of $V$, focusing on the case when $X$ forms a cone over the origin. In this case, we identify $X$ with the corresponding projective variety $X \subseteq \bbP V$ with the induced action of $\GL(V)$. Characterizing $G_X$ is particularly well-motivated in this setting. For example, if $X$ is preserved by a subgroup $H \subseteq \GL(V)$, then each homogeneous component $I_d(X)$ of its defining ideal is a representation of $H$: in this setting representation theory provides powerful tools to determine $I(X)$. If $H=G_X$ is the maximal such subgroup, then $I_d(X)$ will have the fewest number of irreducible components to keep track of. So characterizing $G_X$ allows one to take maximal advantage of the symmetries of $X$.

A common example that we consider is the case where $X$ is a secant variety of some other variety $Y$. If $Y \subseteq \bbP V$, define
\[
\sigma^\circ_r(Y) := \{ x \in \bbP V : x \in \langle y_1 \vvirg y_r \rangle \text{ for some } y_1 \vvirg y_r \in Y\} ;
\]
the $r$-th secant variety of $Y$ is $\sigma_r(Y) = \bar{\sigma^\circ_r(Y)}$, where the closure can be taken equivalently in the Euclidean or the Zariski topology. The inclusion $G_{\sigma_r(Y)} \supseteq G_Y$ always holds. When $G_{\sigma_r(Y)} = G_Y$ we say that the $\sigma_r(Y)$ has the \textit{expected preserver}. However, the inclusion can be strict: an immediate case occurs if $\sigma_r(Y)=\bbP V$, so that $\GL(V)=G_{\sigma_r(Y)} \supsetneq G_Y$.

The tensor setting is of particular interest. Given complex vector spaces $V_1 \vvirg V_k$, let $V=V_1 \otimes \dots \otimes V_k$ be their tensor product and $Y=\bbP V_1 \times \dots \times \bbP V_k \subseteq \bbP V$ be the \textit{Segre variety} of pure tensors $v_1 \otimes \dots \otimes v_k$. The natural action of $\GL(V_1) \ttimes \GL(V_k) $ defines a subgroup $G(V_1 \vvirg V_k) \subseteq \GL(V)$ which clearly preserves $Y$; moreover, permutation of the tensor factors of the same dimension preserves $Y$ as well. This shows $G_{Y} \supseteq G(V_1 \vvirg V_k) \rtimes \frakS$ where $\frakS \subseteq \frakS_k$ is the subgroup of the symmetric group permuting factors of the same dimension. It was proven in~\cite{westwick1967} that in fact equality holds; see also \autoref{prop:seg_ver}.

The representation theory of $G(V_1 \vvirg V_k) \rtimes \frakS$ can be determined using \textit{Schur-Weyl duality} and the representation theoretic approach has been extensively used to determine equations for $\sigma_r(Y)$~\cite{weyman2003cohomology,landsberg2004ideals,landsberg2007ideals,landsberg2012tensors,raicu2012secant,landsberg2013equations}. However, one could take advantage of yet more symmetries if $\sigma_r(Y)$ has a preserver larger than expected. For example, if $Y=\bbP(\bbC^2) \times \bbP(\bbC^{a}) \times \bbP(\bbC^{b})$ and $a \leq r <\min\{b, 2a\}$, then we will observe that $G_{\sigma_r(Y)}=G(\bbC^2 \otimes \bbC^a,\bbC^b)$ instead of the expected $G(\bbC^2 ,\bbC^a,\bbC^b)$, see \autoref{prop:example}. In spite of this example, our main result is that $r$-th secant varieties of Segre varieties have the expected preserver in many cases; namely when there are ``enough" $r \times r$ minors of flattenings that are not identically zero on $\sigma_r(Y)$; we refer to \autoref{sec:prelim} for the relevant definitions.

For an integer $k$, let $[k] = \{ 1 \vvirg k\}$. Following \cite[Definition 3.6]{draisma2011partition}, we say that a collection of bipartitions $\calB$ of $[k]$ is \emph{separating} for $[k]$ if for every $i,j \in [k]$ there exist a bipartition $\{ I,J\} \in \calB$ such that $i \in I , j \in J$. Given $\bfn = (n_1 \vvirg n_k)$, let
\begin{equation*}
s(\bfn) = \max \left\{ \min \left\{ \textstyle \prod _{i \in I} n_i,  \prod _{j \in J} n_j : \{I,J\} \in \calB \right\} : \calB \text{ separating collection for $[k]$}\right\}.
\end{equation*}

\begin{theorem}\label{thm:main}
Let $\bfn = (n_1 \vvirg n_k)$ be a $k$-tuple of positive integers. Let $1\leq r \leq s(\bfn)-1$ be an integer. Let $V_1 \vvirg V_k$ be vector spaces with $\dim V_i = n_i$, let $Y = \bbP V_1 \ttimes \bbP V_k$ and $X = \sigma_r(Y)$. Then $G_{\sigma_r(Y)}={G(V_1 \vvirg V_k) \rtimes \frakS}$, where $\frakS \subseteq \frakS_k$ is the subgroup permuting factors of the same dimension.
\end{theorem}
In particular, $G_{\sigma_r((\bbP^{n-1})^{\times k})} = G(\bbC^{n},\dots,\bbC^n) \rtimes \frakS_k$ for all $r \leq n^{\lfloor k/2 \rfloor}$, see \autoref{cor:secant_segre}. \autoref{thm:main} recovers the classical result of~\cite{westwick1967} characterizing the preservers of Segre varieties, and the result of~\cite{GUTERMAN200061} that secant varieties of two-factor Segre varieties have the expected preserver unless they fill the ambient space. Our proof of this theorem uses the prolongation result of~\cite{sidman2009prolongations}, which characterizes the degree $r+1$ equations for $\sigma_{r+1}(Y)$.

We introduce some mathematical preliminaries in~\autoref{sec:prelim} and introduce several techniques for computing preservers in~\autoref{sec:tech}. We prove~\autoref{thm:main} in~\autoref{sec:secant}. We use the methods we introduced in order to determine the preservers of several other sets of tensors, including subspace varieties in \autoref{sec:subspace}, slice rank 1 tensors and more general partition varieties in \autoref{sec:part}, and symmetric tensors of bounded Waring rank in \autoref{sec:waring}. In~\autoref{sec:quant} we characterize preservers of some sets of quantum states, including {biseparable states} and hyperdeterminantal surfaces. In~\autoref{sec:compute} we use a Lie algebra approach to directly compute preservers of secant varieties in some small cases.

\subsection*{Acknowledgments}
We thank Kangjin Han, Nathaniel Johnston, Joseph Landsberg, Giorgio Ottaviani, and William Slofstra for helpful discussions and for pointing out some important references. B.L. acknowledges that this material is based upon work supported by the National Science Foundation under Award No. DMS-2202782.

\section{Preliminaries}\label{sec:prelim}

In the following, a \emph{variety} is an affine or a projective complex algebraic variety. We refer to \cite{harris2013algebraic} for the definition and the basic results on algebraic varieties that we use in the following. Given a variety $X \subseteq \bbP V$, let $I(X) \subseteq \bbC[V] \simeq \Sym V^*$ denote its ideal: $I(X)$ is a homogeneous saturated ideal in the polynomial ring $\bbC[V]$. If $x \in X$, let $T_x X \subseteq V$ denote the affine tangent space of $X$. For every $x$, one has $\dim X +1 \leq \dim T_x X$ and a point $x \in X$ is called \emph{smooth} if equality holds; if the inequality is strict, $x$ is called \emph{singular}. The set $\Sing(X) \subseteq X$ of singular points is a (Zariski) closed subset of $X$ called the \emph{singular locus} of $X$. For a vector space $V$ and an integer $d$, let $S^d V$ denote the subspace of symmetric tensors in $V^{\otimes d}$, that is those tensors which are invariant under permutation of the tensor factors; the space $S^d V$ can be identified with the space $\bbC[V^*]_d$ of homogeneous polynomials of degree $d$ on the dual space $V^*$. A tensor $T \in V_1 \ootimes V_k$ induces linear maps $T_I : \bigotimes_{i \in I} V_i^* \to \bigotimes_{j \notin I} V_j$ called the \emph{flattening maps} of $T$. We will sometimes drop the index and simply write $T : \bigotimes_{i \in I} V_i^* \to \bigotimes_{j \notin I} V_j$ if no confusion arises.

For a variety $X \subseteq \bbP V$, let $G_X \subseteq \GL(V)$ be its linear preserver subgroup. The group $G_X$ is a closed algebraic subgroup of $\GL(V)$, hence it is a union of finitely many connected (irreducible) components. Let $G_X^\circ$ be the connected component containing the identity element $\id_V$; it is easy to see that $G_X^\circ$ is a normal subgroup of $G_X^\circ$; see e.g. \cite[Ch.I,Sec. 1.2]{Borel:LinAlgGp}. The quotient $G_X / G_X^\circ$ is finite.

Let $\Aut(X)$ be the group of (algebraic) automorphisms of a variety $X$; these are regular bijective maps whose inverse is regular; it is an important fact that $\Aut(X)$ does not depend on the embedding $X \subseteq \bbP V$. Every element $g \in G_X$ induces an automorphism. This gives rise to a map $G_X \to \Aut(X)$; if $X$ is linearly non-degenerate in $\bbP V$, then the kernel of this map is the central subgroup $Z(\GL(V)) \subseteq \GL(V)$ of scalar multiples of the identity.

Given vector spaces $V_1 \vvirg V_k$ with $\dim V_i \geq 2$ and integers $d_1 \vvirg d_k \geq 1$, the space of partially symmetric tensors of multidegree $(d_1 \vvirg d_k)$ is  $S^{d_1} V_1 \ootimes S^{d_k}V_k$. The natural action of $\GL(V_1) \ttimes \GL(V_k)$ on $S^{d_1} V_1 \ootimes S^{d_k}V_k$ defines a group homomorphism
\[
\GL(V_1) \ttimes \GL(V_k) \to \GL(S^{d_1} V_1 \ootimes S^{d_k}V_k)
\]
whose kernel is $\{ (\lambda_1 \id_{V_1} \vvirg \lambda_k \id_{V_k}) : \lambda_1^{d_1} \cdots \lambda_k^{d_k} = 1\}$. Denote by $G(V_1 \vvirg V_k)$ the image of this homomorphism in $\GL(S^{d_1} V_1 \ootimes S^{d_k}V_k)$. We point out that $G(V_1 \vvirg V_k)$, as an abstract group, depends both on the spaces $V_1 \vvirg V_k$ and on the degrees $d_1 \vvirg d_k$.

 The \textit{Segre-Veronese embedding} of $\bbP V_1 \ttimes \bbP V_k$ is the map
\begin{align*}
\nu_{d_1 \vvirg d_k} : \bbP V_1 \ttimes \bbP V_k &\to \bbP (S^{d_1} V_1 \ootimes S^{d_k} V_k) \\
([v_1] \vvirg [v_k]) &\mapsto [v_1^{d_1} \ootimes v_k^{d_k}].
\end{align*}
The image of $\nu_{d_1 \vvirg d_k}$ is called the \textit{Segre-Veronese variety}. Its $r$-th secant variety $\sigma_r( \nu_{d_1 \vvirg d_k} (\bbP V_1 \ttimes \bbP V_k))$ parameterizes partially symmetric tensors in $S^{d_1} V_1 \ootimes S^{d_k} V_k$ admitting a decomposition as sum of $r$ elements of the Segre-Veronese variety (or limits of such).

The following result is classical, and it allows us to recover and generalize the result of \cite{westwick1967}. The proof is a generalization of \cite[Example II.7.1.1]{Hart:AlgGeom}, together with standard facts on the Picard group of product of projective spaces. We outline it in \autoref{appendix:automorphisms} and we refer to \cite{Cox95} for more general results of this type. For a vector space $V$, let $\PGL(V) = \GL(V) / Z(\GL(V))$.

\begin{lemma}\label{lemma: automorphisms of multiP}
Let $V_1 \vvirg V_k$ be vector spaces with $\dim V_i \geq 2$. Then $\Aut(\bbP V_1 \ttimes \bbP V_k) = (\PGL(V_1) \ttimes \PGL(V_k) ) \rtimes \frakS$, where $ \frakS \subseteq \frakS_k$ is the subgroup of the symmetric group permuting factors of the same dimension.
\end{lemma}

We can use \autoref{lemma: automorphisms of multiP} to determine the linear preserver of Segre-Veronese varieties. In the case $d_1 = \cdots = d_k = 1$, this recovers the result of \cite{westwick1967}.
\begin{proposition}\label{prop:seg_ver}
Let $X = \nu_{d_1 \vvirg d_k}(\bbP V_1 \ttimes \bbP V_k) \subseteq \bbP (S^{d_1} V_1 \ootimes S^{d_k}V_k)$ be the Segre-Veronese variety of rank one partially symmetric tensors of multidegree $(d_1 \vvirg d_k)$. Then
\[
G_X = G(V_1 \vvirg V_k) \rtimes \tilde{\frakS}
\]
where $G(V_1 \vvirg V_k)$ denotes the image of the group $\GL(V_1) \ttimes \GL(V_k)$ in $\GL(S^{d_1} V_1 \ootimes S^{d_k}V_k)$ and  $\tilde{\frakS} \subseteq \frakS \subseteq \frakS_k$ is the subgroup permuting factors of the same dimension and the same degree.
\end{proposition}
\begin{proof}
There is a natural embedding of $ G(V_1 \vvirg V_k) \rtimes \tilde{\frakS}$ into $\GL(S^{d_1} V_1 \ootimes S^{d_k}V_k)$ whose image is clearly a subgroup of $G_X$.

The kernel of the natural map $G_X \to \Aut(\bbP V_1 \ttimes \bbP V_k)$ is the central subgroup $Z \subseteq \GL(S^{d_1} V_1 \ootimes S^{d_k}V_k)$ because the Segre-Veronese variety is linearly non-degenerate. This shows that $G_X^\circ = G(V_1 \vvirg V_k)$. Moreover, since $G(V_1 \vvirg V_k)$ surjects onto $\Aut^\circ(\bbP V_1 \ttimes \bbP V_k) = \PGL(V_1) \ttimes \PGL(V_k)$, by~\autoref{lemma: automorphisms of multiP} we have that any element $h \in G_X$ induces, up to composing it with an element of $G(V_1 \vvirg  V_k)$, a permutation in the subgroup $\frakS$ permuting the factors of $X$ of the same dimension.

We show that any such permutation is an element of $\tilde{\frakS}$. Let $h \in G_X$, let $L \subseteq S^{d_1} V_1 \ootimes S^{d_k}V_k$ be a hyperplane and let $h\cdot L$ be its image under the action of $h$. Then $\nu_{d_1 \vvirg d_k}^{-1}(L)$ and $\nu_{d_1 \vvirg d_k}^{-1}(h\cdot L)$ are two hypersurfaces in $\bbP V_1 \ttimes \bbP V_k$ of multidegree $(d_1 \vvirg d_k)$. However, since $h$ induces a permutation in $\frakS$, we have $\nu_{d_1 \vvirg d_k}^{-1}(h\cdot L)$ is a hypersurface of multidegree $(d_{h(1)} \vvirg d_{h(k)})$ as well. This shows that $h$ permutes factors of the same dimension and embedded in the same degree; in other words $h \in G(V_1 \vvirg V_k) \rtimes \tilde{\frakS}$ and this concludes the proof.
\end{proof}

\section{Techniques for computing the linear preserver}\label{sec:tech}
In this section we give an overview of several methods that we will use to compute $G_X$ when $V$ is a complex vector space, $G=\GL(V)$, and $X \subseteq \bbP V$ is a projective variety. The action of $\GL(V)$ on $V$ induces an action on the polynomial ring $\Sym V^*$ by $g \cdot f = f \circ g^{-1}$.

We record a straightforward result, which is however central for most of the techniques of this section.
\begin{lemma}\label{lemma: stabilize ideal}
Let $X \subseteq \bbP V$ be a variety and let $G_X \subseteq \GL(V)$ be its linear preserver. Then
\[
G_X=\{g \in \GL(V) : g\cdot I(X)=I(X)\}.
\]
More precisely, for every $d \geq 1$, $G_X \subseteq \{g \in \GL(V) : g\cdot I_d(X)=I_d(X)\}$ and equality holds for $d \geq d_0$, where $d_0$ is the largest degree of a generator of $I(X)$.
\end{lemma}
\begin{proof}
It suffices to prove that
\[
G_X \subseteq \{g \in \GL(V) : g\cdot I_d(X)=I_d(X)\}
\]
for all $d \in \bbN$, and that
\[
G_X \supseteq \{g \in \GL(V) : g\cdot I_d(X)=I_d(X)\}
\]
for all $d\geq d_0$.

For the first containment, if $g \in G_X$ and $f \in I_d(X)$, then $g \cdot f = f \circ g^{-1}$ vanishes on $g\cdot X = X$; therefore $g \cdot f \in I_d(X)$. For the second containment, if $g \cdot I_d(X)=I_d(X)$ for $d\geq d_0$, then $g^{-1} \cdot I_d(X)=I_d(X)$, so $f(g(x))=0$ for all $x \in X$ and $f \in I_d(X)$. Since the equations of $I_d(X)$ define $X$, it follows that $g \in G_X$.
\end{proof}

The upshot of \autoref{lemma: stabilize ideal} is that the linear preserver of any variety coincides with the linear preserver of some auxiliary linear space, namely any homogeneous component $I_d(X) \subseteq S^d V^*$ for $d$ at least the largest degree of a generator of $I(X)$.

\begin{remark}
The proof of the second containment in \autoref{lemma: stabilize ideal} shows that $G_{I_d} \subseteq G_X$ if $I_d \subseteq S^d V^*$ is a set of equations defining $X$. For the reverse inclusion, one needs $I_d = I_d(X)$ to be the full component of degree $d$ of $I(X)$.

We show with an example that the latter condition is necessary. Let $X = \nu_{1,1}(\bbP V_1 \times \bbP V_2)$ be the variety of rank one matrices. The ideal $I(X)$ is generated by the $2 \times 2$ minors. Let $J_3 \subseteq I_3(X)$ be a generic codimension $1$ linear space and let $J$ be the ideal that it generates. It is easy to see $J$ is a set of equations defining $X$; in fact the saturation of $J$ is $I(X)$. However, a direct calculation shows $G_{J_3}^\circ \subsetneq G_{I_3(X)}^\circ = G(V_1,V_2)$. By semicontinuity, both the saturation statement and the statement regarding the linear preserver can be proved for a specific codimension $1$ linear space. For instance, consider $J_3$ to be generated by all polynomials of the form $x_{i_0j_0}(x_{i_1j_1}x_{i_2j_2} - x_{i_1j_2}x_{i_2j_1})$ except for $x_{11}(x_{22}x_{33} - x_{23}x_{32})$. In this case, the calculation can be performed explicitly.
\end{remark}

\subsection{Lie algebra approach}\label{lie_algebra}
The computation of $G_X$ is often split in two parts: determining the identity component $G_X^\circ$, using standard linear algebra as explained in this section, and then determining $G_X$ using more advanced geometric methods and the fact that it is a finite extension of $G_X^\circ$. In fact, sometimes one can show $G_X = N_{\GL(V)} (G_X^{\circ})$ is the full normalizer of $G_X^\circ$ in $\GL(V)$, which in turn uniquely determines the linear preserver.

Determining $G_X^\circ$ can be reduced to standard linear algebra using \autoref{lemma: stabilize ideal}. Let $\frakgl(V)$ be the Lie algebra of $\GL(V)$. The action of $\GL(V)$ on the spaces $S^d V^*$ induces a Lie algebra action of $\frakgl(V)$, defined by Leibniz rule: in coordinates, for $A = (a_{ij}) \in \frakgl(V)$, and $f \in S^d V^*$, we have
\[
A \cdot f = - \sum_{i,j=1}^n a_{ij} x_j \frac{\partial f}{\partial x_i}.
\]
Let $I(X)$ be generated in degree at most $d$, so that $G_X = G_{I_d(X)}$. Let $\frakg_X \subseteq \frakgl(V)$ be the Lie algebra of $G_X$. Then $\frakg_X$ uniquely determines $G_X^\circ$: in other words, $G_X^\circ$ is the only connected algebraic subgroup of $\GL(V)$ whose Lie algebra is $\frakg_X$.

The Lie algebra $\frakg_X$ is uniquely determined by the condition
\[
\frakg_X=\{A \in \frakgl(V) : A \cdot I_d(X) \subseteq I_d(X)\}.
\]
This follows from the fact that the Lie algebra of the preserver of a linear subspace is equal to the set of endomorphisms preserving that subspace, see e.g.~\cite[Section 13.2]{humphreys2012linear}. In particular, since the Lie algebra action is linear, $\frakg_X$ is identified as the solution set of a linear system. To obtain generators for $G_X^{\circ}$, one can \emph{integrate} the Lie algebra $\frakg_X$, see e.g. \cite[Section 1.3]{Pro:LieGroups}. Usually this is not straightforward and in practice it is more convenient to verify that a natural guess of $G_X^\circ$ has the Lie algebra $\frakg_X$.

Once $G_X^\circ$ is determined, one can hope to verify that the normalizer $N_{\GL(V)} (G_X^{\circ})$ also preserves $X$, in which case one obtains $G_X = N_{\GL(V)} (G_X^{\circ})$. If this is not the case, the task of computing $G_X$ is much more challenging. Some methods are outlined in \cite{Kayal:Affine_proj_poly,GarGur:simple_poly_stab}.

In the case of secant varieties of Segre-Veronese varieties, this last step is immediate if $G_X^\circ = G(V_1 \vvirg V_k)$ is the expected identity component. Indeed, we have the following result, which is a more general version of \cite[Lemma 2.4]{Ges:Geometry_of_IMM}.
\begin{proposition}\label{prop: normalizer}
    Let $V_1 \vvirg V_k$ be vector spaces, let $d_1 \vvirg d_k$ be nonnegative integers and let $H \subseteq \GL(S^{d_1} V_1 \ootimes S^{d_k}V_k)$ be a closed algebraic subgroup such that $G(V_1 \vvirg V_k) \rtimes \tilde{\frakS} \subseteq H$, where $\tilde{\frakS}$ is the subgroup of $\frakS_k$ permuting factors of the same dimension and the same degree. If $H^\circ = G(V_1 \vvirg V_k)$ then $H = G(V_1 \vvirg V_k) \rtimes \tilde{\frakS}$.
\end{proposition}
\begin{proof}
Write $G = \GL(S^{d_1} V_1 \ootimes S^{d_k}V_k)$. Since the identity component $H^\circ$ coincides with $G(V_1 \vvirg V_k)$, we have $H \subseteq N_{G}(G(V_1 \vvirg V_k))$. Hence, conjugation in $H$ defines a map
\[
\phi: H \to \Aut(G(V_1 \vvirg V_k))
\]
whose kernel is the centralizer $C_H(G(V_1 \vvirg V_k))$ of $G(V_1 \vvirg V_k)$ in $H$. Since $S^{d_1} V_1 \ootimes S^{d_k}V_k$ is an irreducible representation for $G(V_1 \vvirg V_k)$, a consequence of Schur's Lemma \cite[Lemma 1.7]{fulton2013representation} is that the centralizer $C_{G}(G(V_1 \vvirg V_k))$ coincides with the center $Z(G) = \{ \lambda \Id : \lambda \in \bbC^\times\}$; in particular $C_H(G(V_1 \vvirg V_k)) \subseteq G(V_1 \vvirg V_k)$.

Consider the composition  
\[
H \xto{\phi} \Aut(G(V_1 \vvirg V_k)) \to \Out(G(V_1 \vvirg V_k)),
\]
where $\Out(G(V_1 \vvirg V_k)) = \Aut(G(V_1 \vvirg V_k)) / \Inn ( G(V_1 \vvirg V_k))$ is the group of outer automorphisms and $\Inn(G(V_1 \vvirg V_k)) = G(V_1 \vvirg V_k) / Z(G)$ is the group of inner automorphisms. The kernel of this composition is the preimage via $\phi$ of the subgroup $\Inn(G(V_1 \vvirg V_k))$, that is the subgroup of $H$ generated by $G(V_1 \vvirg V_k)$ and $\ker (\phi)$. Since $\ker \phi = C_H(G(V_1 \vvirg V_k))$ is contained in $G(V_1 \vvirg V_k)$, we have that the kernel of the composition is $G(V_1 \vvirg V_k)$. Therefore, one obtains a injective map 
\[
\iota: H/ (G(V_1 \vvirg V_k)) \to \Out(G(V_1 \vvirg V_k)).
\]
There is a natural correspondence between $\Out(G(V_1 \vvirg V_k))$ and the symmetries of the Dynkin diagram of $G(V_1 \vvirg V_k)$, see \cite[Prop. D40]{fulton2013representation}. Under this correspondence, \cite[Prop. 2.2, Cor. 2.4]{BeGaLa:Linear_preservers} implies that the map $\iota$ surjects onto the subgroup of those symmetries preserving the marked Dynkin diagram of the representation $S^{d_1} V_1 \ootimes S^{d_k}V_k$. This is exactly the subgroup $\tilde{\frakS}$.

This concludes the proof, because it shows that $H$ has the same number of irreducible components as $G(V_1 \vvirg V_k) \rtimes \tilde{\frakS}$, hence the two groups coincide.
\end{proof}

\subsection{Prolongation and partial derivatives}\label{sec:tech_prolong}

Applying \autoref{lemma: stabilize ideal} in order to determine $G_X$ relies on knowing the ideal $I(X)$, which in general is far from understood. However, the condition $G_X \subseteq G_{I_d(X)}$ for some specific, low degree $d$ is sometimes sufficient to determine $G_X$: indeed, even if $I(X)$ is not generated in degree (at most) $d$, one can often prove directly $G_{I_d(X)}\subseteq G_X$, which in turn yields the equality. A further reduction is obtained by applying the following straightforward result. For an element $\Delta \in S^\delta V^*$ and an element $f \in S^d V$, let $\Delta \contract f \in S^{d-\delta} V$ denote the contraction of $f$ against $\Delta$. We use the same notation for subspaces $E \subseteq  S^\delta V^*$ and $F \subseteq S^d V$: $E \contract F$ is the image of $E \otimes F$ under the contraction map, that is the span of all possible contractions of an elemenet of $F$ against an element of $E$.
\begin{lemma}\label{lemma: derivatives}
 Let $d,d'$ be nonnegative integers with $d' \leq d$. Let $E \subseteq S^d V^*$ and let $E' = S^{d'} V \contract E \subseteq S^{d-d'} V$ be the linear span of all partials of order $d'$ of the elements of $E$. Then $G_E \subseteq G_{E'}$.
\end{lemma}
\begin{proof}
The proof is immediate: if $g \in G_E$, then by definition $g \cdot E = E$; moreover $g \cdot S^{d'} V = S^{d'}V$ simply because $g$ acts linearly on $S^{d'} V$. Therefore
\[
g \cdot E' = g \cdot (S^{d'} V \contract E) = (g \cdot S^{d'} V) \contract (g \cdot E) = S^{d'} V \contract E = E',
\]
where we use in the second equality that the contraction is $\GL(V)$-equivariant. This shows $g \in G_{E'}$.
\end{proof}
Applying \autoref{lemma: derivatives} to the case where $E = I_d(X)$ is the degree $d$ component of the ideal of a variety, we have $G_X \subseteq G_{I_d(X)} \subseteq G_{ S^{d'} V \contract I_d(X)}$.

This method is particularly useful in the setting of secant varieties. To this end, we recall a strong characterization of the low degree components of the ideal of a secant $\sigma_r(Y)$ in terms of the \emph{prolongation} of the ideal of $Y$, following \cite{sidman2009prolongations}: if $A \subseteq S^d V^*$ is a linear space and $p \geq 1$, the $p$-th prolongation of $A$ is defined by
\[
A^{(p)} = \{ f \in S^{d + p} V^* : S^p V \contract f \subseteq A\};
\]
in other words $A^{(p)}$ is the space of forms whose $p$-th partial derivatives belong to $A$. Equivalently $A^{(p)} = (A \otimes S^{p} V^*) \cap S^{d+p} V^*$.
\begin{proposition}[{\cite[Theorem 1.2]{sidman2009prolongations}}]\label{prop: prolongation}
Let $Y \subseteq \bbP V$ be a variety with $I_{d-1}(Y) = 0$. Then $I_{r(d-1) +1}(\sigma_r(Y)) = I_{d}(Y)^{((r-1)(d-1))}$.
\end{proposition}
We use this method to prove our main result,~\autoref{thm:main}, in~\autoref{sec:secant}.

\subsection{Singular locus}\label{sec:sing}
One more useful method to determine $G_X$ consists in determining $G_Z$ for auxiliary varieties $Z$ obtained from $X$ and satisfying the property that $G_X \subseteq G_Z$. One such variety is the singular locus $\Sing(X)\subseteq X$.
\begin{lemma}\label{lem:sing}
 Let $X \subseteq \bbP V$ be a variety. Then
 \[
  G_X \subseteq  G_{\Sing(X)}.
 \]
\end{lemma}
\begin{proof}
The action of $G_X$ on $\bbP V$ preserves the dimension of the affine tangent spaces. Indeed, if $g \in G_X$ and $x \in X$, then $T_{g x} X=g \cdot T_x X$, regarding $T_x X$ as a linear subspace of $V$.
\end{proof}

There are situations where one can determine $G_{\Sing(X)}$ and prove directly the inclusion $G_{\Sing(X)} \subseteq G_X$, which in turn yields equality. For example, we can use the singular locus method to give a short and simple proof that secant varieties of two-factor Segre varieties are as expected until they fill the space. This recovers the result of \cite[Section 3]{GUTERMAN200061}.
\begin{proposition}\label{prop:bipartite}
Let $X=\sigma_r(\bbP V_1 \times \bbP V_2)$ with $r < \min\{\dim V_1,\dim V_2\}$. Then $G_X=G(V_1,V_2)\rtimes \frakS$, where $\frakS=\frakS_2$ if $\dim V_1=\dim V_2$ and is the trivial group otherwise.
\end{proposition}
\begin{proof}
When $r=1$, the statement is true by \autoref{prop:seg_ver}. The general result follows by induction on $r$, using the fact that
\[
\Sing( \sigma_r(\bbP V_1 \times \bbP V_2)) = \sigma_{r-1} ( \bbP V_1 \times \bbP V_2);
\]
this is easy to verify directly, and we refer to \cite[Section II.2]{ArCoGrHa:Vol1} for a complete proof. By \autoref{lem:sing}, we obtain $G_X \subseteq G_{\bbP V_1 \times \bbP V_2} = G(V_1 ,V_2) \rtimes \frakS$. On the other hand, it is clear that $G(V_1 ,V_2) \rtimes \frakS \subseteq G_X$, which concludes the proof.
\end{proof}
The same proof applies to secant varieties of quadratic Veronese embeddings, that is varieties of symmetric matrices of bounded rank.
\begin{proposition}\label{prop:bipartite_symmetric}
Let $X=\sigma_r(\nu_2(\bbP V_1))$ with $r < \dim V_1$. Then $G_X=G(V_1)$.
\end{proposition}
We record here two results which are consequence of classification results for the singular locus of secant varieties in special cases. A great deal of work has been devoted to the study of singular loci of secant varieties: we mention  \cite{Kanev:Chordal,MOZ:SecantCumulants,han2018singularities,khadam2020secant,furukawa2021singular} for the particular case of varieties of tensors, and \cite{coppens2004singular} for the case of curves.

In \cite{coppens2004singular}, the singular locus of $\sigma_r(C)$ is determined in the case where $C \subseteq \bbP ^N$ is an irreducible curve satisfying a fairly strong linear independence condition. This result applies to the rational normal curve which in turn allows us to deduce the following:
\begin{proposition}\label{prop:rnc}
Let $V_1$ be a vector space of dimension $2$ and let $d,r$ be integers with $r \leq d/2$. Let $X = \sigma_r(\nu_d(\bbP V_1)) \subseteq \bbP S^d V_1$. Then $G_X = G(V_1)$.
\end{proposition}
\begin{proof}
    By the main Theorem in \cite{coppens2004singular}, if $C \subseteq \bbP^N$ is an irreducible curve with the property that any $2r$ points on $C$ are linearly independent in $\bbP^N$, then $\Sing(\sigma_r(C)) = \sigma_{r-1}(C)$. This property is satisfied when $C = \nu_d(\bbP^1)$ is the rational normal curve and $r \leq d/2$. An induction argument on $r$ completes the proof.
\end{proof}
The results of \cite{Kanev:Chordal,han2018singularities} allow us to characterize the preserver of the second and third secant varieties of Veronese varieties.
\begin{proposition}\label{prop:small_waring}
Let $V_1$ be a vector space of dimension $n \geq 2$. Let $d \geq 3$, $r = 2,3$ and $X = \sigma_r(\nu_d(\bbP V_1))$. If $(d,n,r) \neq (3,2,2),(3,2,3),(4,2,3)$ then $G_X = G(V_1)$. Otherwise $G_X = \GL(S^dV_1)$.
\end{proposition}
\begin{proof}
Since $G(V_1) = G_{\nu_d(\bbP V_1)}$, the inclusion $G(V_1) \subseteq G_X$ always holds.

For $r=2$, and $(d,n) \neq (3,2)$, \cite[Thm. 3.3]{Kanev:Chordal} states that $\Sing(X) = \nu_d(\bbP V_1)$. Therefore $G_{X} \subseteq G_{\nu_d(\bbP V_1)} = G(V_1)$ showing equality.

If $r=3$, we consider two cases. First suppose $d =3$, $(d,n)=(4,3)$ or $d \geq 5$: in this case \cite[Cor. 2.11 and Thm. 2.1]{han2018singularities} state that $\Sing(X) = \sigma_2(\nu_d(\bbP V_1))$; using again \cite[Thm. 3.3]{Kanev:Chordal}, we deduce $\Sing(\Sing(X)) = \nu_d(\bbP V_1))$. Therefore $G_{X} \subseteq G_{\nu_d(\bbP V_1)} = G(V_1)$ showing equality.

If instead $d=4$ and $n \geq 4$, \cite[Thm. 2.1]{han2018singularities} shows that $\Sing( X ) = \Sub_2$ is the variety of forms which can be written in (at most) two variables after a suitable change of coordinates. The result in this case follows by \autoref{thm: subspace vars}, which guarantees that $G_{\Sub_2} = G(V_1)$ as desired. 

In the case $(d,n,r) = (3,2,r)$,  we have $\sigma_2(\nu_3(\bbP V_1)) = \bbP S^3 V_1$, therefore $G_{\sigma_2(\nu_3(\bbP^1))} =G_{\sigma_3(\nu_3(\bbP^1))}  = \GL(V_1)$. Similarly, in the case $(d,n,r) = (4,2,3)$, we have $\sigma_3(\nu_4(\bbP V_1)) = \bbP S^4 V_1$ and $G_{\sigma_3(\nu_4(\bbP^1))} = \GL(S^4 V_1)$. 
\end{proof}

In \autoref{sec:subspace}, we use the singular locus method to determine the linear preserver of subspace varieties. Moreover, in \autoref{sec:part}, we determine the preserver of the variety of tensors of slice rank $1$ and more generally partition varieties. A straightforward but important fact is the characterization of the singular locus of reducible varieties: if $X = X_1 \cup \cdots \cup X_s$ is a reducible variety, and $X_i$'s are its irreducible components, then
\[
\Sing(X) = \bigcup_{i = 1}^s \Sing(X_i) \cup \bigcup_{i<j} (X_i \cap X_j).
\]
For a positive integer $k$, let $\Sing^{(k)}(X)=\Sing \circ \dots \circ \Sing (X)$ ($k$ times), and set $\Sing^{(0)}(X)=X$.

\subsection{Dual varieties}
Another auxiliary variety which is useful in determining $G_X$ is the \emph{dual variety} of $X$. If $X \subseteq \bbP V$ is a variety, its dual variety is
\[
X^\vee = \bar{ \{ H \in \bbP V^* : T_x X \subseteq H \text{ for some $x \in X \setminus \Sing(X)$}\} }.
\]
A fundamental property of dual varieties is the biduality theorem \cite[Theorem 1.1]{Gelfand1994DiscriminantsRA}: the dual variety of the dual variety of $X$ coincides with $X$, that is ${X^{\vee}}^\vee = X$. Under the natural isomorphism $\GL(V) = \GL(V^*)$, we have the following result.
\begin{proposition}\label{prop:dual}
Let $X \subseteq \bbP V$ be a variety. Then $G_X=G_{X^{\vee}}$.
\end{proposition}
\begin{proof}
Let $g \in G_X$ and let $H \in \bbP V^*$ be a hyperplane tangent to $X$, that is $H \supseteq T_xX$ for some smooth point $x \in X$. Then $g H \supseteq g \cdot T_xX= T_{gx}X$ is tangent to $X$ at the smooth point $gx \in X$. This shows that if $H \in X^\vee$ then $gH \in X^\vee$; in other words $g$ preserves the set of hyperplanes tangent to $X$ at smooth points. Passing to the closure, we deduce that $g$ preserves $X^\vee$, that is $g \in G_{X^\vee}$. We obtain $G_X \subseteq G_{X^\vee}$. Exchanging the roles of $X$ and $X^\vee$, we have the reverse inclusion by the biduality theorem since ${X^{\vee}}^ \vee=X$.
\end{proof}

We use this result in~\autoref{sec:quant} to characterize the linear preservers of hyperdeterminantal surfaces.

\section{Preservers of secant varieties of Segre varieties}\label{sec:secant}

This section is devoted to the proof of \autoref{thm:main} and of other characterizations of the preserver of secant varieties of Segre varieties. We begin characterizing the preserver in the case of pencils of matrices, that is tensors of order $3$ with one factor of dimension $2$. In this case, we characterize the range $r$ for which the preserver is not the expected one.
\begin{proposition}\label{prop:example}
Let $a,b,r$ be positive integers satisfying $2 \leq a \leq r < \min\{b, 2a\}$. Let $Y=\bbP(\bbC^2)\times \bbP(\bbC^a)\times \bbP(\bbC^b)\subseteq \bbP(\bbC^2 \otimes \bbC^a \otimes \bbC^b)$ and let $X = \sigma_r(Y)$. Then $G_{\sigma_r(Y)}=G(\bbC^2 \otimes \bbC^a, \bbC^b)$.
\end{proposition}
\begin{proof}
A consequence of \cite[Theorem 1.1]{landsberg2007ideals} is that in this range one has the equality $\sigma_r(Y)=\sigma_r(\bbP(\bbC^2 \otimes \bbC^a)\times \bbP(\bbC^b))$. The result follows from~\autoref{prop:bipartite}.
\end{proof}
However, our main result \autoref{thm:main} shows that in a wide range of cases secant varieties of Segre varieties have the expected preserver. Recall that, following \cite[Definition 3.6]{draisma2011partition} we say that a collection of bipartitions $\calB$ of $[k]$ is \emph{separating} for $[k]$ if for every $i,j$ there exist a bipartition $\{ I,J\} \in \calB$ such that $i \in I , j \in J$. Given $\bfn = (n_1 \vvirg n_k)$, let
\begin{equation*}
s(\bfn) = \max \left\{ \min \left\{ \textstyle \prod _{i \in I} n_i,  \prod _{j \in J} n_j : \{I,J\} \in \calB \right\} : \calB \text{ separating collection for $[k]$}\right\}.
\end{equation*}
\begin{theorem}[\autoref{thm:main}]\label{thm: main for segre}
Let $\bfn = (n_1 \vvirg n_k)$ be a $k$-tuple of positive integers. Let $1\leq r \leq s(\bfn)-1$ be an integer. Let $V_1 \vvirg V_k$ be vector spaces with $\dim V_i = n_i$, let $Y = \bbP V_1 \ttimes \bbP V_k$ and let $X = \sigma_r(Y)$. Then $G_{\sigma_r(Y)}={G(V_1 \vvirg V_k) \rtimes \frakS}$, where $\frakS \subseteq \frakS_k$ is the subgroup permuting factors of the same dimension.
\end{theorem}
\begin{proof}
Write $V=V_1\otimes \dots \otimes V_k$ and let $\calB=\{I,J\}$ be a bipartition which realizes the maximum in the definition of $s(\bfn)$. For a tensor $T \in V$, the $2 \times 2$ minors of the flattenings $T_{I} : \bigotimes_{i \in I} V_i^* \to \bigotimes_{j \in J} V_j$ are polynomials in the coefficients of $T$, that is elements of $S^2 V^*$. It is a classical fact that $I(Y)$ is generated by such $2 \times 2$ minors: this is a consequence of Kostant's Theorem \cite[Chapter 16]{landsberg2012tensors} and straightforward representation-theoretic considerations; see also \cite[Theorem 3.9]{draisma2011partition}.

By \autoref{prop: prolongation}, we have $I_{r+1}(\sigma_r(Y))=I_2(Y)^{(r-1)}$. We are going to prove that $S^{r-1}(V) \contract I_2(Y)^{(r-1)} = I_2(Y)$. To see this, notice that $I_{r+1}(\sigma_r(Y))$ contains the subspace $M_r \subseteq S^{r+1}V^*$ spanned by size $r+1$ minors of the flattening $T_{I}$. Indeed, writing the flattening maps in coordinates, one observes that the representing matrix has the $n_1 \cdots n_k$ entries of $T$ as coefficients; as a consequence, the $(r-1)$-th order partials of a minor of size $r+1$ are spanned by minors of size $2$; in other words $M_r \subseteq I(Y)^{(r+1)}$. On the other hand, since $s(\bfn) \geq r+1$, every submatrix of size $2$ can be extended to a submatrix of size $r+1$. Moreover, after choosing coordinates, each entry of the matrix representing the flattening map is a coefficient of the tensor, and each coefficient occurs exactly once; in particular the flattening map can be regarded as a generic symbolic matrix, hence any determinant of a submatrix of size $2$ arises as derivative of any of its extensions. This shows
\[
I_2(Y) \supseteq S^{r-1}(V) \contract I_2(Y)^{(r-1)} \supseteq S^{r-1}(V) \contract M_r = I_2(Y).
\]
We deduce $ G_{\sigma_r(Y)} \subseteq  G_Y$ because
\begin{align*}
 G_{\sigma_r(Y)} &\subseteq G_{I_{2}(\sigma_r(Y))^{(r-1)}}\\
& \subseteq G_{S^{r-1}(V) \contract I_2(Y)^{(r-1)}}  \\
& =G_{I_2(Y)} = G_Y,
\end{align*}
where the first inclusion follows from \autoref{lemma: stabilize ideal}. On the other hand, the inclusion $G_Y \subseteq G_{\sigma_r(Y)}$ always holds. By \autoref{prop:seg_ver}, we conclude $G_X = G(V_1 \vvirg V_k) \rtimes \frakS$.
\end{proof}

If $n_1 = \cdots = n_k =n$, we obtain the following consequence, which follows by taking all bipartitions of $[k]$ with one set having size $\lfloor{k/2}\rfloor$.

\begin{corollary}\label{cor:secant_segre}
Let $k \geq 2$ and let $1 \leq r < n^{\lfloor k/2 \rfloor}$ be an integer. Then
\[
G_{\sigma_r(\bbP(\bbC^{n})^{\times k})}={G(\bbC^{n} \vvirg \bbC^n) \rtimes \frakS_k}.
\]
\end{corollary}
We expect \autoref{thm: main for segre} to hold in a wider range and in the symmetric setting as well. In particular, we do not have examples where the preserver is not the expected one except in cases where a phenomonon similar to the one of \autoref{prop:example} occurs: in other words, the only known setting where the preserver is larger than expected is, to the best of our knowledge, when the secant variety of a Segre variety turns out to be a secant variety of a larger Segre variety, with consequentely a larger preserver. 

However, the proof method of \autoref{thm: main for segre} has some clear limitations. In particular, in the (partially) symmetric setting, the standard flattenings are not generic matrices in the sense that their determinantal varieties do not have expected dimension and expected singular and jacobian loci; in these cases, it is hard to control partial derivatives of the linear spaces of minors. In fact, the results of \cite{BK:CactusScheme,BBF:CactusVars} suggest that the situation for the symmetric setting might be very different.

Regarding the range where $r \geq s(n_1 \vvirg n_k)$ there is a similar problem. Most of the known equations for higher secant varieties are determinantal and they are built via generalized flattening methods such as the Young flattenings of \cite{landsberg2013equations}. However, the resulting determinantal ideals are not generic, and it is hard to control their jacobian loci. 

Two explicit cases beyond the range of \autoref{thm: main for segre} are considered in \autoref{sec:compute}: they are $\sigma_3(\nu_3(\bbP^2)) \subseteq \bbP S^d \bbC^3$, and $\sigma_4(\bbP^2 \times \bbP^2\times \bbP^2) \subseteq \bbP (\bbC^3 \otimes \bbC^3 \otimes \bbC^3)$. In these cases, the Young flattening construction gives rise to a single determinantal equation, and an explicit calculation can show this has the expected preserver.

\section{Preservers of subspace varieties}\label{sec:subspace}

Subspace varieties are subvarieties of tensors with prescribed multilinear ranks, see \cite[Sec. 7.1]{landsberg2012tensors}. In this section, we study them in the setting of partially symmetric tensors. We characterize their singular locus and use such characterization to determine the linear preserver.

Let $V_1 \vvirg V_k$ be vector spaces and let $d_1 \vvirg d_k, r_1 \vvirg r_k$ be positive integers with $r_j \leq \dim V_j$. The \textit{subspace variety} $\setft{Sub}_{r_1,\dots, r_k} \subseteq \bbP(S^{d_1} V_1 \otimes \dots \otimes S^{d_k} V_k)$ of multilinear ranks $r_1 \vvirg r_k$ is
\[
\Sub_{r_1 \vvirg r_k} = \{ T \in S^{d_1} V_1 \ootimes S^{d_k} V_k : \rank(T : V_i^* \to S^{d_i-1} V_i \otimes \textstyle \bigotimes_{j\neq i} S^{d_j}V_j) \leq r_i\}.
\]
If $T \in \Sub_{r_1 \vvirg r_k}$, let $E_i' = \image( T : S^{d_i-1} V_i^* \otimes \textstyle \bigotimes_{j\neq i} S^{d_j}V_j^* \to V_i )$. Then $\dim E_i' \leq r_i$ and $T \in S^{d_1} E_1' \ootimes  S^{d_1} E_k'$.

\begin{remark}\label{rmk: rs one}
Note that $\Sub_{1 \vvirg 1} = \nu_{d_1 \vvirg d_k} (\bbP V_1 \ttimes \bbP V_k)$ is the Segre-Veronese variety. More generally, if $r_i = 1$ then $\Sub_{r_1 \vvirg r_k}$ is a product of a Veronese variety and a subspace variety on a smaller number of factors. For instance, if $r_k = 1$ then $\Sub_{r_1 \vvirg r_{k-1},1} = \Sub_{r_1 \vvirg r_{k-1}} \times \nu_{d_k}(\bbP V_k)$ in the Segre embedding of $\bbP(S^{d_1} V_1 \ootimes S^{d_{k-1}}V_{k-1}) \times \bbP S^{d_k} V_k$ in $\bbP( S^{d_1} V_1 \ootimes S^{d_k} V_k)$.
\end{remark}

In particular, \autoref{rmk: rs one} shows that the singularities of the subspace variety only depend on the multilinear ranks $r_1 \vvirg r_k$ which are not equal to $1$. In the multilinear setting, that is when $d_1 = \cdots = d_k = 1$, the singular locus is characterized in  \cite[Theorem 3.1]{landsberg2007ideals}.

The multilinear case with two factors is special: If $r_1 \leq r_2$ then the subspace variety $\Sub_{r_1,r_2} \subseteq \bbP (V_1 \otimes V_2)$ coincides with $\Sub_{r_1,r_1}$. We observe that this is essentially the only case of \emph{redundancy} for the subspace variety:
\begin{lemma}\label{lemma: nonredundancy}
    Let $V_1 \vvirg V_k$ be vector spaces, let $d_1 \vvirg d_k , r_1 \vvirg r_k$ be nonnegative integers with $r_i \leq \dim V_i$. Suppose $d_1 \geq 2$, or $d_1 = 1$ and $r_1 \leq \dim (S^{d_2 } \bbC^{r_2} \ootimes S^{d_k} \bbC^{r_k})$. Then
    \[
    \Sub_{r_1 - 1, r_2 \vvirg r_k} \subsetneq     \Sub_{r_1, r_2 \vvirg r_k}.
    \]
\end{lemma}
\begin{proof}
    The proof is straightforward. It suffices to exhibit an element in $U = \Sub_{r_1, r_2 \vvirg r_k} \setminus \Sub_{r_1 - 1, r_2 \vvirg r_k}$. If $d_1 \geq 2$, let $T_1 = v_1^{\otimes d_1} + \cdots + v_{r_1}^{\otimes d_1} \in S^{d_1} V_1$ for linearly independent vectors $v_1,\dots,v_{r_1}$, and let $T_2$ for any nonzero element in $S^{d_2} E_2 \ootimes S^{d_k} E_k$, for choices $E_j \subseteq V_j$ with $\dim E_j = r_j$. Then $T = T_1 \otimes T_2$ is an element of $U$.

    If $d_1 = 1$ and $r_1 \leq \dim (S^{d_2 } \bbC^{r_2} \ootimes S^{d_k} \bbC^{r_k})$, let $T_1 \vvirg T_{r_1}$ be linearly independent elements in $S^{d_2} E_2 \ootimes S^{d_k} E_k$ for some choices of $E_j \subseteq V_j$ with $\dim E_j = r_j$. Let $v_1 \vvirg v_{r_1} \in V_1$ be linearly independent. Then $T = \sum_1^{r_1} v_i \otimes T_i$ is an element of $U$.
\end{proof}
Following \autoref{lemma: nonredundancy}, we say that a $k$-tuple of multilinear ranks $(r_1 \vvirg r_k)$ is \emph{non-redundant} (for the space $S^{d_1} \bbC^{n_1} \ootimes S^{d_k} \bbC^{n_k}$) if, for every $i$, either $d_i \geq 2$ or $r_i \leq \prod_{j \neq i} \binom{r_j+d_j-1}{d_j}$.

We characterize the singular locus of subspace varieties in general in \autoref{prop: singularsubspace}. First we record a technical result, which is useful in the following.
\begin{lemma}\label{lemma: tangentsubspace}
Let $T \in \Sub_{r_1 \vvirg r_k} \subseteq \bbP S^{d_1} V_1 \ootimes S^{d_k}V_k$. Let
    \[
    E_j' =  \image ( T: (S^{d_i-1} V_i^* \otimes \textstyle \bigotimes_{j\neq i} S^{d_j}V_j^*) \to V_i).
    \]
Let $E_j \supseteq E_j'$ be a space with $\dim E_j = r_j$. Then the linear space
    \[
    L(T) = S^{d_1}E_1 \ootimes S^{d_k}E_k + \sum_{i=1}^k  ( \frakgl(V_i) \cdot T )
    \]
    is contained in the affine tangent space to $\Sub_{r_1 \vvirg r_k} $ at $T$.

    If $E_i = E_i'$ for every $i$, then
    \[
   \dim L(T) =  \textstyle \prod_{j=1}^k \binom{r_j+d_j-1}{d_j} + \sum_{j=1}^k r_j(n_j-r_j) .
    \]
\end{lemma}
\begin{proof}
  The terms $\frakgl(V_j) \cdot T$ are the image of $T$ via the Lie algebra action $\frakgl(V_j)$. In other words the space $\sum_{i=1}^k  ( \frakgl(V_i) \cdot T )$ is the tangent space to the $G(V_1 \vvirg V_k)$-orbit of $T$. Since $\Sub_{r_1 \vvirg r_k}$ is closed under the action of $G(V_1 \vvirg V_k)$, this tangent space is contained in $T_T \Sub_{r_1 \vvirg r_k}$.

  The space $S^{d_1}E_1 \ootimes S^{d_k}E_k$ is entirely contained in $\Sub_{r_1 \vvirg r_k}$, so it is contained into the tangent space at any of its points. Since $T \in S^{d_1}E_1 \ootimes S^{d_k}E_k$, we obtain that $T_T \Sub_{r_1 \vvirg r_k} \supseteq L(T)$.

For the dimension count, consider a complement $F_j \subseteq V_j$ of $E_j$, so that $V_j = E_j \oplus F_j$ and correspondingly $V_j^* = F_j^\perp \oplus E_j^\perp$. Then
\[
\frakgl(V_j) \simeq E_j \otimes F_j^\perp \ \oplus \ E_j \otimes E_j^\perp \ \oplus \ F_j \otimes F_j^\perp \ \oplus \ F_j \otimes E_j^\perp.
\]
The linear space $\frakgl(V_j) \cdot T$ is the sum of the four linear spaces arising by the contraction against each of the four summands. Since $T \in S^{d_1} E_1 \ootimes S^{d_k} E_j$, we have $E_j^\perp \contract T = 0$; in particular the second and fourth term give no contribution. Moreover $(E_j \otimes F_j^\perp) \contract T$ is contained in the space $S^{d_1} E_1 \ootimes S^{d_k} E_k$, hence its contribution to $L(T)$ is redundant. Finally $\dim ( (F_j \otimes F_j^\perp ) \contract T ) = r_j (n_j-r_j)$ since $T$ is concise in $S^{d_1} E_1 \ootimes S^{d_k} E_j$.

All these terms give rise to terms contained in distinct summands of the direct sum decomposition of $S^{d_1} V_1 \ootimes S^{d_k} V_k$ induced by the splittings $V_i = E_i \oplus F_i$. The dimension calculation follows by adding the dimension of the direct summands.
\end{proof}

Using \autoref{lemma: tangentsubspace}, we can characterize the singular locus of subspace varieties, via a generalization of the argument in the proof of \cite[Theorem 3.1]{landsberg2007ideals}.
\begin{proposition}\label{prop: singularsubspace}
    Let $V_1 \vvirg V_k$ be vector spaces with $\dim V_i = n_i \geq 2$. Let $d_1 \vvirg d_k, r_1 \vvirg r_k$ be non-negative integers such that $(r_1 \vvirg r_k)$ is non-redundant and $r_i < n_i$. Then
\[
\Sing( \Sub_{r_1 \vvirg r_k} ) = \bigcup_{i = 1 \vvirg k} \Sub_{r_1 \vvirg r_i-1 \vvirg  r_k}   .
\]
\end{proposition}
\begin{proof}
Let $\Theta = \bigcup_{i = 1 \vvirg k} \Sub_{r_1 \vvirg r_i-1 \vvirg  r_k}$.

We first show that if $T \in \Sub_{r_1 \vvirg r_k} \setminus \Theta$ then $T$ is smooth. This is done via an incidence correspondence argument. For every $j$ let $\Gr(r_j, V_j)$ be the Grassmannian of $r_j$-planes in $V_j$. Define
\begin{align*}
\calS = &\bar{\{ (E_1 \vvirg E_k, T) : T \in \bbP (S^{d_1} E_1 \ootimes S^{d_k} E_k) \}}
\end{align*}
which is a subvariety of the product $\prod_{j=1}^k \Gr(r_j, V_j) \times \bbP (S^{d_1} V_1 \ootimes S^{d_k}V_k)$.

Let $\pi_{\Gr} : \calS \to \prod_{j=1}^k \Gr(r_j, V_j)$ and $\pi_V : \calS \to \bbP (S^{d_1} V_1 \ootimes S^{d_k}V_k)$ be the two projections. The map $\pi_{\Gr}$ defines a projectivized vector bundle on $\prod_{j=1}^k \Gr(r_j, V_j)$: the fiber over $(E_1 \vvirg E_k)$ is $\bbP (S^{d_1} E_1 \ootimes S^{d_k} E_k)$. In particular $\calS$ is smooth and it has dimension
\begin{align*}
\dim \calS = &  \dim \textstyle \prod_{j=1}^k \Gr(r_j, V_j) + \dim \bbP S^{d_1} \bbC^{r_1} \ootimes S^{d_k} \bbC^{r_k}\\
& =\textstyle \sum_{j=1}^k r_j(n_j-r_j) + \prod_{j=1}^k \binom{r_j+d_j-1}{d_j} -1.
\end{align*}
By definition, the projection $\pi_V$ surjects onto $\Sub_{r_1 \vvirg r_k}$. If $T \in \Sub_{r_1 \vvirg r_k} \setminus \Theta$, the fiber $\pi_V^{-1}(T)$ is (as a set) the single point $(E_1 \vvirg E_k, T)$ where
\[
E_i = \image ( T: (S^{d_i-1} V_i^* \otimes \textstyle \bigotimes_{j\neq i} S^{d_j}V_j^*) \to V_i).
\]
One can compute the differential of $\pi_V$ at this point and verify that it is injective; this calculation is similar to the one in the proof of \autoref{lemma: tangentsubspace}; see also the proof of \cite[Theorem 3.1]{landsberg2007ideals}. This guarantees that the points of $\Sub_{r_1 \vvirg r_k} \setminus \Theta$ are smooth, namely $\Sing( \Sub_{r_1 \vvirg r_k}) \subseteq \Theta$. In fact, a dimension count shows that the affine tangent space $T_T \Sub_{r_1 \vvirg r_k}$ coincides with the space $L(T)$ of \autoref{lemma: tangentsubspace}.

To show the reverse inclusion suppose $T \in \Theta$ and without loss of generality assume $T \in \Sub_{r_1-1,r_2 \vvirg r_k}$; in particular $r_1 \geq 2$, otherwise $r_1-1 = 0$ and $\Sub_{r_1-1,r_2 \vvirg r_k} = \emptyset$; moreover, because of the non-redundancy condition, $\Sub_{r_1-1,r_2 \vvirg r_k} \subsetneq \Sub_{r_1,r_2 \vvirg r_k}$. Let $E_j,E_j'$ as in \autoref{lemma: tangentsubspace}; in particular $E_1$ can be chosen among all subspaces of dimension $r_1$ containing $E_1'$. All such choices give rise to a space as the one described in \autoref{lemma: tangentsubspace} contained in $T_{T} \Sub_{r_1 \vvirg r_k}$. This guarantees $S^{d_1} V_1 \otimes S^{d_2}E_2 \ootimes S^{d_k}E_k \subseteq T_{T} \Sub_{r_1 \vvirg r_k}$ for all such choices of $E_2 \vvirg E_k$. A dimension count, similar to the one in \autoref{lemma: tangentsubspace} shows that the tangent space at $T$ has dimension larger than at a generic point of $\Sub_{r_1,r_2 \vvirg r_k}$. We deduce that $T$ is singular and this concludes the proof.
\end{proof}

We use the method outlined in \autoref{lem:sing} to characterize the linear preserver of subspace varieties:
\begin{theorem}\label{thm: subspace vars}
    Let $V_1 \vvirg V_k$ be vector spaces with $\dim V_i = n_i$. Let $d_1 \vvirg d_k,r_1 \vvirg r_k$ be integers such that $(r_1 \vvirg r_k)$ is non-redundant and $r_i < n_i$.  Then $G_{\Sub_{r_1 \vvirg r_k}} = G(V_1 \vvirg V_k) \rtimes \tilde{\frakS}$ where $\tilde{\frakS}$ is the subgroup of $\frakS_k$ permuting factors with the same dimension and embedded in the same degree.
\end{theorem}
\begin{proof}
By \autoref{prop: singularsubspace}, we have
\[
\Sing( \Sub_{r_1 \vvirg r_k} ) = \bigcup_{i = 1 \vvirg k} \Sub_{r_1 \vvirg r_i-1 \vvirg  r_k}   .
\]
Because of the non-redundancy condition, the union on the right-hand side is a proper subvariety of $\Sub_{r_1 \vvirg r_k}$. Moreover, after possibly replacing each $(r_1 \vvirg r_{i-1},r_i-1,r_{i+1} \vvirg r_k)$ with a non-redundant one, the right hand side can be written as a union over a set of non-redundant multilinear ranks. The result then follows by an induction argument and \autoref{prop:seg_ver} since $\Sub_{1 \vvirg 1} = \nu_{d_1 \vvirg d_k}(\bbP V_1 \ttimes \bbP V_k)$.
\end{proof}

We conclude this section observing that the hypothesis of \autoref{thm: subspace vars} is not tight. For instance, in the case of tensors of order $3$, a similar argument as above shows that if $(r_1,r_2,r_3)$ are non-redundant multilinear ranks with $n_1 = r_1$, $r_i < n_i$ for $i=2,3$, then the preserver of $\Sub_{r_1,r_2,r_3}$ is the expected one. More generally, under some non-trivial arithmetic condition on the $n_i$'s and $d_j$'s, one may allow one of the multilinear ranks to be equal to the dimension of the corresponding tensor factors. We expect this to be sharp. For instance, in the three factors case, if $r_1 = n_1$, $r_2 = n_2$, then the preserver contains $G(V_1 \otimes V_2, V_3)$ which is larger than expected.

\section{Preservers of partition varieties}\label{sec:part}
For a collection of bipartitions $\calB$ of $[k]$, let $\setft{Part}(\calB) \subseteq \bbP(V_1 \otimes \dots \otimes V_k)$ be the set of tensors $T$ for which $T_I : \bigotimes_{i \in I} V_i^* \rightarrow \bigotimes_{j \in J} V_j$ has rank $1$ for some $\{I,J\} \in \calB$. We call $\setft{Part}(\calB)$ the \textit{partition variety} of $\calB$. Special cases of partition varieties include some well-studied varieties of tensors. For example, $\setft{Part}(\{\{\{i\},[k]\setminus \{i\}\}: i \in [k]\})$ is the set of tensors of \textit{slice rank} one, in the sense of \cite{SawTao:NoteSliceRank}. Also, $\setft{Part}(\{\text{all bipartitions of $[k]$}\})$ is the set of tensors of \textit{partition rank} one, in the sense of \cite{naslund2020partition}; these are also known as \textit{biseparable tensors}.

In this section we characterize the linear preservers of partition varieties. We define the intersection of two set-partitions $Q=\{Q_1,\dots, Q_{p}\}$ and $R=\{R_1,\dots, R_q\}$ of $[k]$ to be the set-partition $Q \cap R=\{Q_i \cap R_j : i \in [p], j \in [q]\}$, which is itself a set-partition of $[k]$. For a collection of bipartitions $\calB$, let $\cap \calB$ be their intersection. As an example, note that $\calB$ is separating for $[k]$ if and only if $\cap \calB$ is the set partition consisting of $k$ singletons.

\begin{theorem}
Let $\calB$ be a collection of bipartitions of $[k]$, let $\{ T_1 \vvirg T_m\} = \cap \calB$. Let $V_1 \vvirg V_k$ be vector spaces and, for every $p$, let $\Delta_p = [[ \dim V_i : i \in T_p ]]$ be the multiset of dimensions of the factors occurring in $T_p$. Then
\[
G_{\setft{Part}(\calB)}=G\big(\otimes_{i_1\in T_1} V_{i_1},\dots, \otimes_{i_m \in T_m} V_{i_m}\big)\rtimes \hat{\frakS},
\]
where $\hat{\frakS}$ is the stabilizer of $\{ \Delta_1 \vvirg \Delta_m\}$ with respect to the induced action of $\frakS_k$ on the set of all multisets with elements in $\{ \dim V_1 \vvirg \dim V_k\}$.
\end{theorem}
\begin{proof}
Let $X=\setft{Part}(\calB)$ and let $\calB=\{\{I_1^{(1)}, I_1^{(2)}\},\dots, \{I_{\ell}^{(1)},I_{\ell}^{(2)}\}\}$. Then the Segre reembeddings of $\bbP(\otimes_{j \in I_i^{(1)}} V_j)\times \bbP(\otimes_{j \in I_i^{(2)}} V_j)$ are the irreducible components of $X$. Since each component is smooth, the singular locus of $X$ is the union of pairwise intersections:
\[
\Sing(X)=\bigcup_{\substack{i,i' \in [\ell] \\ i<i'}} \left[(\bbP(\otimes_{j \in I_i^{(1)}} V_j)\times \bbP(\otimes_{j \in I_i^{(2)}} V_j)) \cap (\bbP(\otimes_{j \in I_{i'}^{(1)}} V_j)\times \bbP(\otimes_{j \in I_{i'}^{(2)}} V_j))\right].
\]
Note that each term in the union is again irreducible and smooth. Via an induction argument, one obtains
\begin{align*}
\Sing^{(\ell-1)}(X)&=\bigcap_{i \in [\ell]} (\bbP(\otimes_{j \in I_i^{(1)}} V_j)\times \bbP(\otimes_{j \in I_i^{(2)}} V_j))\\
&= \bbP \big(\otimes_{i_1 \in T_1} V_{i_1}\big)\times \dots \times \bbP \big(\otimes_{i_m \in T_m} V_{i_m}\big).
\end{align*}
By~\autoref{lem:sing}, $G_X$ must preserve $\Sing^{(\ell-1)}(X)$. By \autoref{prop:seg_ver}, we have
\[
G_X \subseteq G_{\Sing^{(\ell-1)}(X)} = G\big(\otimes_{i_1\in T_1} V_{i_1},\dots, \otimes_{i_m \in T_m} V_{i_m}\big)\rtimes \hat{\frakS}.
\]
The reverse inclusion is clear.
\end{proof}

We obtain the following corollary, which in particular applies to the set of tensors of slice rank one and the one of tensors of partition rank one.
\begin{corollary}\label{cor:part}
Let $\calB$ be a collection of bipartitions separating for $[k]$. Let $V_1 \vvirg V_k$ be vector spaces. Then $G_{\Part(\calB)} =G(V_1 \vvirg V_k) \rtimes \frakS$, where $\frakS \subseteq \frakS_k$ is the subgroup permuting factors of the same dimension.

In particular, this is the linear preserver for the variety $\Part(   \{ \{i\},\{i\}^c\}\} : i \in [k])$ of tensors of slice rank one and of the variety $\Part( \textup{all bipartitions of }[k])$ of tensors of partition rank one.
\end{corollary}

\section{Waring rank preservers}\label{sec:waring}

In this section we characterize the preserver of $\sigma_r^{\circ}(\nu_d(\bbP V))$ whenever $d \geq 2r-1$, see \autoref{thm:waring}. This set is not Zariski closed, and correspondingly the methods in this section are more algebraic and combinatorial in nature, rather than purely geometric.~\autoref{thm:waring} should be compared to~\autoref{prop:rnc}, which shows that in the case of secants to the rational normal curve, $\sigma_r(\nu_d(\bbP^1))$ has the expected preserver whenever $d \geq 2r$. See also~\autoref{prop:small_waring}, which characterizes the preservers of border Waring rank $2$ and $3$.

We will use the following technical result.

\begin{lemma}\label{lemma:secant_plane_line}
Let $r$ and $d$ be integers satisfying $d \geq 2r-1$, let $V$ be a vector space with $\dim V = n \geq 2$, and let $L\subseteq \bbP S^d V $ be a line such that $L\subseteq \sigma_r^{\circ}(\nu_d(\bbP V))$. Then $L$ is contained in an $r$-secant plane to $\nu_d(\bbP V)$.
\end{lemma}
\begin{proof}
If $d = 2r-1$ and $n=2$ there is no line $L$ satisfying the hypothesis. Indeed, in this case $\sigma_r(\nu_d(\bbP^1))=\bbP S^d \bbC^2$, and $\sigma_r^{\circ}(\nu_d(\bbP^1))\subseteq \bbP S^d \bbC^2$ is the complement of a hypersurface, see e.g. \cite[Proposition 19]{BucHanMelTei:LocusPtsHighRank}. In particular, every line $L \subseteq \bbP S^d \bbC^2$ intersects such hypersurface. Thus, we may assume that either (i) $d\geq 2r$, or (ii) $d=2r-1$ and $L$ is not contained in a copy of $\bbP S^d \bbC^2$ in $\bbP S^d V$.

Suppose $L = \langle f,g \rangle$. Without loss of generality, the Waring rank of $f$ is exactly $r$, so that
\begin{align*}
f &= v_1^d+\dots+v_r^d, \\
g &= w_1^d+\dots+w_{r'}^d
\end{align*}
for some $r'\leq r$ and $v_1,\dots, v_r,w_1,\dots, w_{r'} \in V\setminus\{0\}$. To complete the proof, we will show that $[w_i] \in \{[v_1],\dots,[v_r]\}$ for all $i =1 \vvirg r'$; this guarantees $L \subseteq \langle v_1^d \vvirg v_r^d \rangle$ which is an $r$-secant plane to $\nu_d(\bbP V)$.

Suppose by contradiction that $[w_1] \notin \{[v_1],\dots,[v_r]\}$. It is easy to see that $v_1^r,\dots,v_r^r,w_1^r $ are linearly independent as elements of $S^r V$, see e.g. \cite[Exercise 3.1]{harris2013algebraic}. Let $\phi_1,\dots, \phi_{r+1}\in S^r(V^*)$ be elements such that
\[
\phi_i \contract v_j^r = \delta_{ij}, \qquad  \phi_i \contract w_1^{r} = \left\{
\begin{array}{ll}
0 & \text{if $i \neq r+1$} \\
1 & \text{if $i = r+1$}
\end{array} \right.
\]
For each $t \in \bbC$ let
\[
h_t= f + t g = v_1^d+\dots+v_r^d+t(w_1^d+\dots+w_{r'}^d).
\]
Therefore $h_t \in L \subseteq \sigma_r^{\circ}(\nu_d(\bbP V))$ for every $t \in \bbC$, and
\begin{align*}
\phi_1 \contract h_t &= v_1^{d-r}+t y_1\\
\phi_2 \contract h_t &= v_2^{d-r}+t y_2\\
&\;\;\vdots\\
\phi_r \contract h_t &=v_r^{d-r}+ty_r\\
\phi_{r+1} \contract h_t &= t(w_1^{d-r}+y_{r+1})
\end{align*}
for some elements $y_1,\dots, y_{r+1} \in S^{d-r}(V)$ depending on the contraction of the $\phi_j$ against $g$.

Note that $v_1^{d-r},\dots, v_r^{d-r},w_1^{d-r}$ are linearly independent. To prove this statement, we proceed similarly to before. In case (i), that is $d \geq 2r$, then $d-r \geq r$ and the statement follows again from \cite[Exercise 3.1]{harris2013algebraic}. In case (ii), $d = 2r-1$ and $L$ is not contained in a copy of $\bbP S^d \bbC^2$; this guarantees that $\{v_1,\dots, v_r,w_1\}$ are not contained in a $\bbP^1 \subseteq \bbP V$, and in this case the statement follows, for instance, from~\cite[Corollary 18]{lovitz_petrov_2023}.

This guarantees that, as points of $\bbP S^{d-r} V$, the elements $\phi_1 \contract h_t \vvirg  \phi_{r+1} \contract h_t$ are independent at $t=0$, and therefore they are independent for generic $t \neq 0$. Fix $t \neq 0$ for which they are independent. This yields $r+1$ linearly independent elements in the space of order $r$ partials of $h_t$ with $t \neq 0$; this implies that $h_t$ has rank at least $r+1$, which is a contradiction since $h_t \in L \subseteq \sigma_r^\circ(\nu_d(\bbP V))$.

In turn this contradiction implies $[w_1] \in \{[v_1],\dots,[v_r]\}$. The same argument holds for every $w_i$, and this concludes the proof.
\end{proof}

\begin{theorem}\label{thm:waring}
Let $r$ and $d$ be integers satisfying $d \geq 2r-1$, let $V$ be a vector space with $\dim V \geq 2$. Then the linear preserver of $\sigma_r^{\circ}(\nu_d(\bbP V))$ in $\GL(S^d V)$ is $G(V)$, namely the image of $\GL(V)$ in $\GL(S^d V)$.
\end{theorem}
\begin{proof}
We are going to prove that $G_{\sigma_r^{\circ}(\nu_d(\bbP V))}\subseteq G_{\sigma_{r-1}^{\circ}(\nu_d(\bbP V))}$. An induction argument provides then $G_{\sigma_r^{\circ}(\nu_d(\bbP V))} \subseteq G_{\nu_d(\bbP V)}$, and we conclude using \autoref{prop:seg_ver}.

Suppose toward contradiction that there exists $g \in G_{\sigma_r^{\circ}(\nu_d(\bbP V))} \setminus G_{\sigma_{r-1}^{\circ}(\nu_d(\bbP V))}$ and let $y \in \sigma_{r-1}^{\circ}(\nu_d(\bbP V))$ be an element with the property that $g y \notin \sigma_{r-1}^{\circ}(\nu_d(\bbP V))$.

Since $y \in \sigma_{r-1}^{\circ}(\nu_d(\bbP V))$, for every $v \in \bbP V$, the line $\langle v^d , y \rangle$ is contained in $ \sigma_{r}^{\circ}(\nu_d(\bbP V))$. Since $g \in G_{\sigma_r^{\circ}(\nu_d(\bbP V))}$, the same holds for the line $g \cdot \langle v^d , y \rangle = \langle g  v^d , g y \rangle$. Since $\langle g  v^d , g y \rangle$ is a line in $\sigma_{r}^{\circ}(\nu_d(\bbP V))$, by \autoref{lemma:secant_plane_line} there exist $u_1 \vvirg u_r \in \bbP V$ such that
\[
\langle g  v^d , g y \rangle \subseteq \langle u_1^d \vvirg u_r^d\rangle.
\]

In particular $u_1^d \vvirg u_r^d$ give a decomposition of $g y$ of length (at most) $r$; since $gy \notin \sigma_{r-1}^{\circ}(\nu_d(\bbP V))$, the decomposition has length exactly $r$. Since $d \geq 2r-1$, such decomposition is unique: this is a consequence of classical results by Sylvester \cite{Sylv:PrinciplesCalculusForms}; see \cite[Theorem 6.4]{Chi:HilbertFunctionsTensorAn} for the exact statement. The uniqueness of the decomposition shows in particular that $\langle u_1^d \vvirg u_r^d\rangle$ is independent of the choice of $v \in \bbP V$, which implies $\nu_d(\bbP V) \subseteq \langle u_1^d \vvirg u_r^d\rangle$. Since the Veronese variety is linearly non-degenerate, we obtain a contradiction and this concludes the proof.
\end{proof}
As discussed in \autoref{sec:secant}, we expect the preserver to be the expected one when passing to the closure $\sigma_r(\nu_d(\bbP^n)) = \sigma_r^\circ(\nu_d(\bbP^n))$ as well. For (very) small $r$, one can study generalizations of \autoref{lemma:secant_plane_line}, characterizing lines in  $\sigma_r(\nu_d(\bbP^n))$ depending on their intersection with the different components of $\sigma_r(\nu_d(\bbP^n)) \setminus \sigma_r^\circ(\nu_d(\bbP^n))$. The classification is far from straightforward already for $r \geq 4$, see e.g. \cite{BalBer:Strat}, and, in different but related settings, it non-trivial even for $r =2$ \cite{GS:secantGrassmannian,Gal:SecantsGeneralized}.

\section{Linear preservers in quantum information theory}\label{sec:quant}

Characterizing linear preservers has found utility in quantum information theory, where the physical quantum operations (completely positive, trace preserving maps) are linear~\cite{johnston2011characterizing,fovsner2013linear,contreras2019resource}. In this section we use our results to characterize preservers of some sets of quantum states/tensors of low entanglement.

\subsection{Preservers of hyperdeterminantal surfaces and the ``onion" of multipartite entanglement}

Let $V=V_1 \otimes \dots \otimes V_k$ be a complex tensor product space and let $X=\bbP V_1 \times \dots \times \bbP V_k \subseteq \bbP V$ be the Segre variety. In \cite{miyake2003classification}, the author proposes a classification of multipartite entanglement in terms of the ``onion"
\[
\bbP(V_1 \otimes \dots \otimes V_k)\supseteq X^{\vee} \supseteq \Sing(X^{\vee}) \supseteq \Sing^{(2)}(X^{\vee}) \supseteq \dots.
\]
\autoref{prop:seg_ver},~\autoref{lem:sing}, and~\autoref{prop:dual} combined tell us the linear preserver at each level:
\begin{theorem}
If $\ell \geq 0$ and $\Sing^{(\ell)}(X^{\vee})\neq \emptyset,$ then $G_{\Sing^{(\ell)}(X^{\vee})}=G(V_1,\dots, V_k) \rtimes \frakS$, where $\frakS \subseteq \frakS_k$ is the subgroup permuting factors of the same dimension.
\end{theorem}

\subsection{Biseparable and fully separable preservers}
Let $V_1,\dots, V_k$ be complex vector spaces of finite dimension endowed with positive definite Hermitian inner products $\langle -, -\rangle$ which we take to be antilinear in the second argument. Let $V=V_1\otimes \dots \otimes V_k$ be endowed with the natural induced inner product. The inner product defines an antilinear isomorphism $V \simeq V^*$ given by $ v^* = \langle - , v \rangle$.  A \textit{quantum channel} is a map $\Phi \in \End(\End(V))$ that is completely positive and trace preserving, see~\cite{watrous2018theory} for details. We consider only \textit{unitary channels}, that is channels of the form $\Phi(A)=U A U^*$ for a unitary $U \in \setft{U}(V)$. A \textit{quantum state} is an endomorphism $\rho \in \End(V)$ that is positive semidefinite and of trace one. We say that $\rho$ is \textit{fully separable} if
\[
\rho \in \setft{conv}\{vv^* : [v]\in \bbP V_1 \times \dots \times \bbP V_k\},
\]
where $\conv(-)$ denotes the convex hull. We say that $\rho$ is \textit{biseparable} if
\[
\rho \in \setft{conv}\{vv^* : [v]\in \setft{Part}(\{\text{all bipartitions of $[k]$}\})\}.
\]
Here, $\setft{Part}(\{\text{all bipartitions of $[k]$}\})$ is the set of biseparable tensors defined in~\autoref{sec:part}.

The recent work~\cite{contreras2019resource} introduces a resource theory of entanglement based on what they call \textit{fully-separable preservers (FSP)} and \textit{biseparable preservers (BSP)}. By definition, these are the quantum channels that preserve the set of fully separable states and biseparable states, respectively.~\autoref{cor:part} and~\autoref{prop:seg_ver} imply the following characterization of unitary FSP's and BSP's. The FSP version of this result was also observed in~\cite[Theorem 6.1]{johnston2011characterizing}.

\begin{theorem}
Let $V=V_1\otimes \dots \otimes V_k$. A unitary channel $\Phi(A)=U A U^* \in \End(\End(V))$ is in BSP/FSP if and only if $U \in (\setft{U}(V_1)\times \dots \times \setft{U}(V_k)) \rtimes \frakS$, where $\setft{U}(V_i)\subseteq \GL(V_i)$ is the unitary subgroup and $\frakS \subseteq \frakS_k$ is the subgroup permuting factors of the same dimension.
\end{theorem}
\begin{proof}
We prove the statement for BSP. The proof for FSP is analogous. Since $\Phi$ is in BSP, for every tensor $v\in V$ of partition rank 1 it holds that $\Phi(vv^*)=(Uv)(Uv)^*$ is biseparable. Since states of the form $uu^*$ are extreme points, it follows that $Uv$ has partition rank 1. So $U$ preserves the partition rank 1 tensors.  This completes the proof by~\autoref{cor:part}.
\end{proof}

\section{Linear preservers computed directly}\label{sec:compute}

In this section we use the Lie algebra approach introduced in~\autoref{lie_algebra} to directly compute linear preservers of some particular secant varieties. The Lie algebra $\frakgl(V)$ acts on every homogeneous component $S^d V^*$ of $\bbC[V]$ by the Leibniz rule:
\[
A \cdot f =  \sum_{i,j=1}^{n} a_{ij} x_i \frac{\partial}{\partial x_j} f
\]
where $x_1 \vvirg x_{n}$ is any basis of $V^*$ and $A = (a_{ij})$ is the represented as a matrix in this fixed basis.

As discussed in~\autoref{lie_algebra}, for a variety $X \subseteq \bbP(V)$ whose ideal $I(X)$ is generated in degree at most $d$, the Lie algebra of $G_X$ is given by
\[
\frakg_X=\{A \in \frakgl(V) : A \cdot I_d(X) \subseteq I_d(X)\}
\]
and it uniquely determines the identity component $G_X^{\circ}$.

In the cases of interest $X=\sigma_r(\nu_{d_1 \vvirg d_k}(\bbP V_1 \ttimes \bbP V_k))$ is a secant variety of a Segre-Veronese variety. In this case, $G(V_1,\dots, V_k) \rtimes \tilde{\frakS} \subseteq G_X$, where $\tilde{\frakS} \subseteq \frakS_k$ is the subgroup permuting factors of the same dimension and embedded in the same degree, and we want to check whether equality holds.

The Lie algebra approach allows us to determine, using elementary linear algebra, whether $\dim(\frakg_X)=\sum_{i=1}^k ((\dim V_i)^2-1)+1 = \dim G(V_1,\dots, V_k)$. If this equality of dimension holds, necessarily $\frakg_X = \frakgl(V_1 \vvirg V_k)$, hence $G_X^\circ = G(V_1,\dots, V_k)$. By \autoref{prop: normalizer} this implies $G_X = G(V_1 \vvirg V_k) \rtimes \tilde{\frakS}$.

Explicitly, the calculation of $\frakg_X$ is reduced to solving a linear system. Fix bases $f_1,\dots, f_p$ of $I_d(X) \subseteq S^d(V^*)$ and $\theta_1,\dots, \theta_q $ of $I_d(X)^{\perp} \subseteq S^d(V)$. Then the condition $A I_d(X) \subseteq I_d(X)$ is equivalent to the system of linear equations
\[
\theta _i \contract A f_j = 0 \quad \text{ for every $i,j$}.
\]
This is usually a very large linear system. However, it is often the case that a suitable choice of basis makes it much more tractable. For instance, often the elements $f_1 \vvirg f_p$ can be chosen to be supported only on a special subset of monomials. This allows one to speed up the computation of $\theta_1,\dots, \theta_q$ by first computing a basis for the complement to $\langle f_1,...,f_p \rangle$ in the space of monomials on which they are supported, and then augmenting it with the remaining monomials. In all of our examples, the $f_i$ have integer coefficients, which allows us to use symbolic computation.

Some examples where this approach was implemented are available at the GitHub repository~\cite{daniels_code}. The implementation is symbolic in \texttt{Macaulay2} \cite{M2} for small examples and numerical in python using the NumPy package \cite{numpy} for the larger ones. In each of the following cases, we verify that the Lie algebra $\frakg_X$ has the expected dimension, and hence $X$ has the expected preserver:
\begin{itemize}
\item ${\sigma}_2(\bbP^{1} \times \bbP^1 \times \bbP^1 \times \bbP^1)$
\item ${\sigma}_3(\bbP^{1} \times \bbP^1 \times \bbP^1 \times \bbP^1)$
\item ${\sigma}_2(\bbP^{1} \times \bbP^1 \times \bbP^1 \times \bbP^2)$
\item ${\sigma}_2(\bbP^{1} \times \bbP^1 \times \bbP^2 \times \bbP^2)$
\item ${\sigma}_2(\bbP^{1} \times \bbP^1 \times \bbP^1 \times \bbP^1 \times \bbP^1)$.
\end{itemize}
The method requires one to start with a set of generators for $I(X)$. In fact, following \autoref{lemma: stabilize ideal}, for the relevant bound, it suffices to consider any homogeneous component $I_d(X)$. The ideal of each $\sigma_2$ is generated in degree $3$ by the $3 \times 3$ minors of the flattenings~\cite{raicu2012secant}. The degree $4$ component of the ideal of ${\sigma}_3(\bbP^{1} \times \bbP^1 \times \bbP^1 \times \bbP^1)$ is spanned by any two of the three determinants of the $4 \times 4$ flattenings \cite{Segre:FormeQuadrilineari}.

All of the above examples are subsumed by~\autoref{thm:main}, but they are useful as concrete implementations and indicators for the size of problems that can be computed. In fact, these and many other examples were computed during the development of this project before we were able to prove \autoref{thm:main}.

Many additional examples can be computed, also beyond the range of \autoref{thm:main}. In the case when $X = \sigma_r(Y)$ is a hypersurface, or more generally it has a single equation of minimal degree, then $\frakg_X$ is the direct sum of the line spanned by $\id_V$ and the kernel of the map $\frakgl(V) \to S^d V^*$ defined by the Lie algebra action applied to the unique generator $f$ of $I_d(X)$ for some $d$. There are a number of cases where this yields a proof that $G_X$ is the expected preserver. Some examples are the following:
\begin{itemize}
    \item $\sigma_3( \nu_3(\bbP^2))$ is a hypersurface of degree $4$ defined by Aronhold's invariant \cite[Section 3.10.1]{landsberg2012tensors}.
    \item $\sigma_4(\bbP^2 \times \bbP^2 \times \bbP^2)$ is a hypersurface of degree $9$ defined by Strassen's invariant \cite[Section 3.9.3]{landsberg2012tensors}.
    \item $\sigma_5(\bbP^1 \times \bbP^1 \times \bbP^1 \times \bbP^1 \times \bbP^1)$ is a variety of codimension $2$, which has a single equation in degree $6$. This example is further discussed below.
\end{itemize}
The case $X = \sigma_5( {\bbP^1 }^{\times 5})$ was studied extensively in \cite{OedSam:5thSecant5P1s}. In \cite{OedSam:5thSecant5P1s} a numerical proof is given of the fact that its ideal is generated by two elements $I(X) = (F_6,F_{16})$ of degree $6,16$ respectively. The proof involves numerical computations but the fact that $I_6(X) = \langle F_6\rangle$ can be proved explicitly using representation theory. We call $F_6$ the Oeding-Sam invariant. By \autoref{lemma: stabilize ideal}, we obtain that $G_X \subseteq G_{F_6}$. Equality follows because one can explicitly verify that $\dim \frakg_{F_6} = 16 = \dim G(V_1 \vvirg V_5)$; therefore $G_X^\circ = G(V_1 \vvirg V_5)$ and $G_X = G(V_1 \vvirg V_5) \rtimes \frakS_5$ by \autoref{prop: normalizer} and \autoref{prop:seg_ver}.

In~\autoref{app:interpolation}, we illustrate the method used in \cite{daniels_code} to compute $F_6$, Aronhold's invariant, and Strassen's invariant for these computations.

\appendix

\section{Autormorphism group of multiprojective space}\label{appendix:automorphisms}

Let $X \subseteq \bbP V$ be an algebraic variety. A vector bundle $\calL \to X$ of rank one is called a line bundle on $X$. Line bundles, up to isomorphism, form an abelian group with the operation of tensor product; this is the Picard group of $X$, denoted $\Pic (X)$. We refer to \cite[Chapter 3]{Shaf:BasicAlgGeom1} for some fundamental results about the Picard group.

If $f : X \to Y$ is a morphism of varieties, then $f$ induces, via pullback, a group homomorphism $f^* : \Pic(Y) \to \Pic(X)$ where the line bundle $\calL \to Y$ is mapped to its pullback $f^* \calL \to X$. If $f$ is an isomorphism, then $f^*$ is an isomorphism as well.

In particular, every element $g \in \Aut(X)$, induces an automorphism $g^* : \Pic(X) \to \Pic(X)$. In the case of product of projective spaces $X = \bbP V_1 \ttimes \bbP V_k$ this allows one to characterize $\Aut(X)$ and obtain \autoref{lemma: automorphisms of multiP}. We provide a sketch of the proof here, assuming minimal background. We start with two basic results.
\begin{lemma}\label{lem: picard}
Let $V_1 \vvirg V_k$ be complex vector spaces with $\dim V_i \geq 2$. Then
\[
\Pic(\bbP V_1 \ttimes \bbP V_k) = \bbZ \cdot \calH_1 \ooplus \bbZ \cdot \calH_k,
\]
where $\calH_j = \pi_j^* \calO_{\bbP V_j}(1)$ is the pullback of the hyperplane bundle of $\bbP V_j$ via the projection of $\pi_j$ of $\bbP V_1 \ttimes \bbP V_d$ onto its $j$-th factor.
\end{lemma}
\begin{proof}
 Since $X = \bbP V_1 \ttimes \bbP V_k$ is smooth, the correspondence between line bundles and divisors defines an isomorphism between $\Pic(X)$ and its first Chow group $\rmCH^1(X)$. The result follows by characterizing this Chow group; we refer to \cite[Section 2.1.4]{EisHar:3264} for an explicit proof in the case of two factors. The general case is a straightforward generalization.
\end{proof}

\begin{lemma}\label{lem: pgl}
Let $V$ be a complex vector space. Then
    \[
    \Aut(\bbP V) = \PGL(V) = \GL(V) / \bbC^\times
    \]
where $\bbC^\times = \{ \lambda \id_V : \lambda \neq 0\}$ is the subgroup of scalar matrices.
\end{lemma}
\begin{proof}
Let $f : \bbP V \to \bbP V$ be an automorphism. The Picard group of $\bbP V$ is $\Pic(\bbP V) = \bbZ$ and its two generators are $\calO(\pm 1)$, that is the hyperplane bundle and its dual. The pull-back map $f^*: \Pic(X) \to \Pic(X)$ is an isomorphism so it maps $\calO(1)$ to either $\calO(1)$ or $\calO(-1)$. However, $\dim H^0(f^*\calO(1)) = \dim H^0(\calO(1)) = \dim V$ whereas $\dim H^0(\calO(-1)) = 0$ so $f^*\calO(1) = \calO(1)$. This implies that $f$ maps hyperplanes to hyperplanes and more generally linear spaces to linear spaces. This concludes the proof.
\end{proof}

Using these results, we provide a proof of \autoref{lemma: automorphisms of multiP}. Following \autoref{lem: picard}, write $\calO(a_1 \vvirg a_k ) = \calH_1^{\otimes a_1} \ootimes \calH_k^{\otimes a_k}$ where $ \calL^{\otimes a}$ denotes the $a$-th tensor power of the line bundle $\calL$ if $a$ is nonnegative or the $(-a)$-th tensor power of the dual bundle $\calL^\vee$ is $a$ is negative.
\begin{lemma}[\autoref{lemma: automorphisms of multiP}]
Let $V_1 \vvirg V_k$ be vector spaces with $\dim V_i \geq 2$. Then
\[
\Aut(\bbP V_1 \ttimes \bbP V_k) = (\PGL(V_1) \ttimes \PGL(V_k) ) \rtimes \frakS.
\]
where $ \frakS \subseteq \frakS_d$ is the subgroup of the symmetric group permuting factors of the same dimension.
\end{lemma}
\begin{proof}
    Let $X = \bbP V_1 \ttimes \bbP V_k$ and let $f : X \to X$ be an automorphism. Then $f^* : \Pic(X) \to \Pic(X)$ is an automorphism of $\bbZ^k$; it is represented by a $k \times k$ matrix $A$ with integer entries and determinant $\pm 1$. Suppose without loss of generality $\dim V_1 \leq \dim V_i$ for every  $i$. We are going to show $f^* \calH_1 = \calH_j$ for some $j$ such that $\dim V_j = \dim V_1$. The full proof follows by induction on $k$.

We have $f^* \calH_1 = \calO(a_1 \vvirg a_k)$ where $a_1 \vvirg a_k$ are the entries of the first column of the matrix $A$. If $a_j$ is negative for any $j$, then $H^0(\calO(a_1 \vvirg a_k)) = 0$. This guarantees that $a_j$ is nonnegative for every $j$ because $\dim H^0(\calH_1) = \dim V_1$. Moreover $\dim H^0(\calO(a_1 \vvirg a_k)) > \dim V_1$ unless $\calO(a_1 \vvirg a_k) = \calH_j$ for some $j$ such that $\dim V_j = \dim V_1$. This proves the desired condition and in turn shows that the $\calH_j$'s with $\dim V_j = \dim V_1$ are permuted by $f^*$. Proceeding by induction completes the proof.
\end{proof}

\section{Interpolation for the computation of invariants}\label{app:interpolation}

We illustrate the method used in \cite{daniels_code} to compute the relevant invariants for the computations of \autoref{sec:compute}. The implementation is in \texttt{Macaulay2}. Most invariants relevant for the study of secant varieties have a geometric interpretation, and in order to check whether they vanish on a given tensor one usually does not compute the invariant symbolically, but rather checks the corresponding geometric property. However, in order to compute the Lie algebra annihilator in \autoref{sec:compute}, we need symbolic expressions of the three invariants of interest: Aronhold's invariant of degree $4$ for $\sigma_3(\nu_3(\bbP^2))$, Strassen's invariant of degree $9$ for $\sigma_4(\bbP^2 \times \bbP^2 \times \bbP^2)$, and the Oeding-Sam invariant of degree $6$ for $\sigma_5(\bbP^1 \times \bbP^1 \times \bbP^1\times \bbP^1\times \bbP^1)$.

In the case of Aronhold's invariant, we use its description as any of the degree $4$ pfaffians of a special $9 \times 9$ skew-symmetric matrix, see \cite[Section 3.10.1]{landsberg2012tensors} and \cite[Theorem 1.2]{Ottav09}. An analogous description is possible for Strassen's invariant, but it results in the symbolic computation of a determinant of size $9$, which is beyond the computational power we had at our disposal. Moreover, we do not have a determinantal interpretation for the Oeding-Sam invariant. For these two cases, we resorted to a standard interpolation method, enhanced with the use of representation theory.

In both settings, our goal is to compute an element $F \in S^d ({\bbC^{n}}^{\otimes k})$ which is invariant for the action of $\SL_n \ttimes \SL_n$. Importantly, $F$ vanishes on a variety $X$ from which one can easily sample points.

Suppose $S^d (\bbC^{n} \ootimes \bbC^n)$ is endowed with a basis $M_1 \vvirg M_N$, for instance the basis of the monomials of degree $d$ in the coefficients of the tensors. A naive interpolation approach consists in sampling randomly $N$ points from $X$ and constructing a matrix of size $N \times N$ whose entry $(i,j)$ is the evaluation of $M_j$ at the $i$-th sample point. With high probability, this matrix has rank $N-1$ and the unique element of the kernel is the column vector of the coefficients of $F$ in the basis $M_1 \vvirg M_N$. This approach is computationally intractable because of the size of $N$. Therefore, we use the invariancy for the action of $\SL_n \ttimes \SL_n$ to reduce the size of the search space, and then use interpolation on a small subspace of $S^d (\bbC^{n} \ootimes \bbC^n)$. In order to reduce the search space, we use mainly two methods: passing to the weight zero space, and symmetrizing by the action of a finite group.

Let $\bbT = \bbT_n^{\SL} \ttimes \bbT_n^{\SL}$ be the torus of $k$-tuples of diagonal matrices having determinant one. Every monomial $M_i \in S^d (\bbC^{n} \ootimes \bbC^n)$ is rescaled by the action of $\bbT$ by a monomial $\omega(M_i)$ in the diagonal elements of $\bbT$; the exponent vector of $\omega(M_i)$ is called weight of $M_i$. If $F$ is invariant for the action of $\SL_n \ttimes \SL_n$, then it is necessarily invariant for the action of $\bbT$, hence it must be linear combination of monomials such that $\omega(M_i) = 1$. The span of these monomials is the \emph{weight zero} subspace of $S^d (\bbC^{n} \ootimes \bbC^n)$. Computing a basis of the weight zero subspace significantly reduces the size of the search space; effectively, this is done by endowing the polynomial ring $\bbC[\bbC^{n} \ootimes \bbC^n]$ with a modified grading that records the weight of each variable.

Further symmetrization tricks are possible. For instance, each factor $\SL_n$ contains a copy of the Weyl group $\frakS_n$. For every weight zero monomial $M_i$, one can consider the symmetrization $C_i = \sum_{\sigma \in \frakS_n} \sigma(M_i)$. The invariant $F$ is linear combination of these symmetrizations, which usually span a proper linear subspace of the weight zero space.

Finally, in the cases of interest for us, we know that Strassen's invariant and the Oeding-Sam invariant are \emph{skew} with respect to the action of the symmetric group $\frakS_k$ which permutes the tensor factors. By skew-symmetrizing a spanning set of the weight zero subspace, one can further reduce the size of the search space.

This is the procedure implemented in \cite{daniels_code} to compute Strassen's invariant and the Oeding-Sam invariant. We illustrate the method explicitly to compute a much simpler invariant, that is the $3 \times 3$ determinant polynomial, namely the equation of $\sigma_2(\bbP^2 \times \bbP^2) \subseteq \bbP(\bbC^3 \otimes \bbC^3)$. This polynomial is invariant for the action of $\SL_3 \times \SL_3$ and for the action of $\frakS_2$ by transposition.

\begin{example}
    Let $V_1 = V_2 = \bbC^3$ and let $\det_3 \in S^3 (V_1 \otimes V_2)$ be the $3 \times 3$ determinant polynomial. It is invariant for the action of $\SL(V_1) \times \SL(V_2)$ and by transposition. Its standard expression in terms of the entries of a symbolic matrix $T = (t_{i,j})$ is the sum of six monomials and we have
    \begin{align*}
    \det_3 ( g T h ) &= \det_3(T) \text{ for every $g , h \in \SL_3$} \\
    \det_3 ( T^\bft) &= \det_3(T)
    \end{align*}
where $\cdot ^\bft$ denotes the transpose matrix.

The torus $\bbT $ is the subgroup of $\SL_3 \times \SL_3$ defined by
\[
\bbT = \left\{
\left( \left(\begin{array}{ccc} a_1 & & \\ & a_2 & \\ & & a_1^{-1}a_2^{-1}
\end{array}\right) , \left(
\begin{array}{ccc} b_1 & & \\ & b_2 & \\ & & b_1^{-1}b_2^{-1}
\end{array}\right) \right) : a_1,a_2,b_1,b_2 \in \bbC \setminus\{0\}\right\},
\]
and it acts by rescaling every monomial in $\bbC[t_{ij} : i,j = 1,2,3]$ by a Laurent monomial in $a_1,a_2,b_1,b_2$: for instance, if $(g,h) \in \bbT$, then $(g,h) \cdot t_{12} = a_1b_2 \ t_{12}$ and $(g,h) \cdot t_{13}t_{21} = (a_1 b_1^{-1}b_2^{-1})(a_2b_1)t_{13}t_{21}  = a_1a_2b_2^{-1}t_{13}t_{21}$.

The weights induce a $(\bbZ \times \bbZ^4)$-grading on the polynomial ring $\bbC[t_{ij}: i,j=1,2,3]$ such that $\deg(t_{ij})$ is the vector in $\bbZ^5$ whose first component is $1$ and the other four components define the weight of $t_{ij}$, that is the exponent vector of the monomial which rescales $t_{ij}$; for instance, $\deg(t_{12}) = (1,1,0,0,1)$ and $\deg(t_{13}t_{21}) = (2,1,1,0,-1)$. The degree $3$ weight zero space is the homogeneous component of degree $(3,0,0,0,0)$, which is a subspace of $S^3(V_1 \otimes V_2)$.

Note that $\dim S^3(V_1 \otimes V_2) = 165$ whereas the weight space has dimension $6$, and it is spanned by the six monomials appearing in $\det_3$.

In this case, symmetrizing by the Weyl action of $\frakS_3 \times \frakS_3 \subseteq \SL_3 \times \SL_3$ trivializes the problem, because $\det_3$ is the only weight zero invariant under this action. In general however it is a convenient way to further reduce the search space.

Finally, since the determinant is invariant under transposition, one can reduce the search space only considering the span of the symmetrizations of the six weight zero monomials.
\end{example}

\bibliographystyle{alphaurl}
\bibliography{tensor_rank.bib}

\end{document}